\documentclass[10pt,reqno]{amsart}
\usepackage{amsmath}
\usepackage{amssymb}
\usepackage{amsthm,comment}
\usepackage{mathabx,mathrsfs}
\usepackage{url}
\usepackage{setspace}     \spacing{1}
\usepackage{enumitem}
\usepackage{bbm}
\usepackage{times}

\numberwithin{equation}{section}

\usepackage{caption}  %\captionsetup[table]{skip=pt}
\usepackage{makecell}
\usepackage{float}
\restylefloat{table}

\usepackage{tablefootnote}

\usepackage{color}
\usepackage[dvipsnames]{xcolor}

\renewcommand{\emph}[1]{{\bf #1}}

\DeclareMathOperator{\intd}{d}

\usepackage[colorlinks=true,linkcolor=blue,citecolor=blue,linktocpage=true,urlcolor=black]{hyperref} %\hypersetup{urlcolor=blue, citecolor=red}
    \usepackage[abbrev,nobysame]{amsrefs}
    \renewcommand{\MR}[1]{}

\usepackage[capitalize]{cleveref}

\usepackage[toc, title]{appendix}
 
\usepackage{caption}
\usepackage[labelformat=simple]{subcaption}
\newtheorem{theorem}{Theorem} [section]
\newtheorem*{theorem*}{Theorem}

\newtheorem{lemma}[theorem]{Lemma}

\theoremstyle{definition}

\newtheorem{definition}[theorem]{Definition}
\newtheorem{question}[theorem]{Question}

\newtheorem{example}[theorem]{Example}

\newtheorem{remark}[theorem]{Remark}
\newtheorem{notation}{Notation}

\usepackage[colorinlistoftodos]{todonotes}

\theoremstyle{remark}
\newlist{enumlemma}{enumerate}{3} 
\setlist[enumlemma]{label*={ (\alph*)}, ref= {(\alph*)} }
\newlist{enumcount}{enumerate}{3} 
\setlist[enumcount]{label*={ (\arabic*)}, ref= {(\arabic*)} }

\usepackage{marginnote}

\newcommand{\diff}{\mathrm{Diff}}
\newcommand{\Diff}{\diff}

\newcommand{\Sl}{\mathrm{SL}}
\def\SL{\Sl}

\newcommand{\norm}[1]{\left\lVert#1\right\rVert}
\DeclareFontFamily{U}{wncy}{}
\DeclareFontShape{U}{wncy}{m}{n}{<->wncyr10}{}
\DeclareSymbolFont{mcy}{U}{wncy}{m}{n}
\DeclareMathSymbol{\Sh}{\mathord}{mcy}{"58}

\renewcommand{\phi}{\varphi}

\newcommand{\Rcal}{ {\mathcal R}}

\newcommand{\Rbb}{{\mathbb R}}
\newcommand{\Zbb}{{\mathbb Z}}

\title[Partially hyperbolic lattice actions on $2-$step nilmanifolds]{Partially hyperbolic lattice actions on $2$-step nilmanifolds}
\author[H.~Lee]{Homin Lee}
\author[S.~Sandfeldt]{Sven Sandfeldt}
\address{Northwestern University, Evanston, IL 60208, USA}
\email{homin.lee@northwestern.edu}
\address{Department of mathematics, Kungliga Tekniska högskolan, Lindstedtsvägen 25, SE-100 44 Stockholm, Sweden.}
\email{svensan@kth.se}
\date{\today}

\long\def\symbolfootnote[#1]#2{\begingroup\def\thefootnote{\fnsymbol{footnote}}
\footnote[#1]{#2}\endgroup}

\usepackage{bbm}

\begin{document}
%\symbolfootnote[0]{\it Preliminary version.  Updated \today. }% \\ Send comments to \href{mailto:awb@northwestern.edu}{\normalfont \nolinkurl { awb@northwestern.edu}}}

\let\oldtocsection=\tocsection
\let\oldtocsubsection=\tocsubsection 
\let\oldtocsubsubsection=\tocsubsubsection
 
\renewcommand{\tocsection}[2]{\hspace{0em}\oldtocsection{#1}{#2}}
\renewcommand{\tocsubsection}[2]{\hspace{1em}\oldtocsubsection{#1}{#2}}
\renewcommand{\tocsubsubsection}[2]{\hspace{2em}\oldtocsubsubsection{#1}{#2}}
\setcounter{tocdepth}{1}

%%%%%%%%%%
%\documentclass[12pt]{amsart}
%\usepackage[top=80pt,bottom=80pt,left=110pt,right=110pt]{geometry}
%\usepackage{graphicx} % Required for inserting images
%\usepackage{amsmath}
%\usepackage{enumitem}
%\usepackage{amssymb}
%\usepackage{xcolor}
%\usepackage{tikz}
%\usepackage{amsthm}
%\usepackage{hyperref}

%\newtheorem{theorem}{Theorem}[section]
%\newtheorem{notation}{Notation}[section]

%\newtheorem{proposition}{Proposition}[section]
%\newtheorem{lemma}{Lemma}[section]
%\newtheorem{corollary}{Corollary}[section]
%\newtheorem{claim}{Claim}[section]

%\newtheorem{mainTheorem}{Theorem}
%\renewcommand*{\themainTheorem}{\Alph{mainTheorem}}

%\theoremstyle{definition}
%\newtheorem{definition}{Definition}[section]
%\newtheorem{example}{Example}[section]
%\newtheorem{remark}[theorem]{Remark}

%\counterwithin{equation}{section}
%\renewcommand\qedsymbol{$\blacksquare$}
%\renewcommand{\baselinestretch}{1.10} 

%\setlength{\parindent}{0cm}
%\setlength{\parskip}{0.4cm}

%\newcommand{\norm}[1]{\left\lVert#1\right\rVert}
%\newcommand{\intd}{{\rm d}}
%\newcommand{\interior}[1]{%
 % {\kern0pt#1}^{\mathrm{o}}%
%}

\begin{abstract}
We prove global rigidity results for actions of higher rank lattices on nilmanifolds containing a partially hyperbolic element. We consider actions of higher rank lattices on tori or on $2-$step nilpotent nilmanifolds, such that the actions contain a partially hyperbolic element with $1-$dimensional center. In this setting we prove, under a technical assumption on the partially hyperbolic element, that any such action must be by affine maps. This extends results from \cite{BrownRodriguez-HertzWang2017} to certain lattice actions that are not Anosov.
\end{abstract}

\maketitle
\tableofcontents
\section{Introduction}

Let $G$ be a connected semisimple Lie group without compact factors and with finite center. Moreover, assume that every simple factor of $G$ has $\mathbb{R}-$rank at least $2$. Let $\Gamma$ be a lattice in $G$. We call such $G$ a \textbf{higher rank semisimple Lie group} and call such a lattice $\Gamma\leq G$ a \textbf{higher rank lattice}. Note that there are many other semisimple Lie groups with $\mathbb{R}-$rank at least $2$, such as $\SL(2,\Rbb)\times \SL(2,\Rbb)$, we will, however, not consider these groups as higher rank groups (since they have simple factors of $\mathbb{R}-$rank $1$) in this paper. 

For a higher rank lattice $\Gamma$, many rigid phenomena are expected to hold, and many have been proved. One such rigidity phenomenon is \textbf{Margulis' superrigidity}, which says that any finite dimensional representation of $\Gamma$ can be obtained by restricting a representation of $G$ (up to a compact error), see for instance, \cites{Mar91,MorrisArith}. A $\Gamma-$action on a smooth manifold can be viewed as a non-linear representation of $\Gamma$, so, motivated by Margulis' superrigidity, it is natural to try classifying smooth $\Gamma-$actions. This research program was initiated by Zimmer and is now known as \textbf{the Zimmer program}, see \cites{Fishersuv, Brownsurv}. One of the main topics in the Zimmer program is the classification of smooth $\Gamma-$actions on manifolds, when one, or more, elements of the action have some special dynamical property. One of the main questions in this direction is if higher rank lattice actions with some hyperbolicity are algebraic (see, for instance, \cite{Gorodnik2007}*{Conjecture 5}):
\begin{question}\label{Question:MainQuestion}
{\it Let $M$ be a closed, smooth manifold and $\alpha:\Gamma\to{\rm Diff}^{\infty}(M)$ a smooth $\Gamma-$action with $\gamma_{0}\in\Gamma$ such that $f = \alpha(\gamma_{0})$ is a partially hyperbolic diffeomorphism\footnote{A diffeomorphism $f$ is partially hyperbolic if the tangent space splits into a contracting part, an expanding part and a center part which is dominated by the stable and the unstable directions, see Section \ref{SubSec:FiberedPH}.}. Is $\alpha$ smoothly conjugated to an algebraic action? More precisely, is $M$ diffeomorphic to a homogeneous space, and $\alpha$ smoothly conjugated to an action by affine maps on this homogeneous space?}
\end{question}

Question \ref{Question:MainQuestion} can be divided into two separate parts;
\begin{enumerate}
    \item if a manifold $M$ admits a higher rank lattice action $\alpha$ that contains a partially hyperbolic element, does it carry a homogeneous structure?
    \item can one classify all higher rank lattice actions on homogeneous spaces when the action contains a partially hyperbolic element?
\end{enumerate}

In this paper, we focus on the second part of the question. In the extreme case, when we assume that the action contains an Anosov diffeomorphism (so there is no center direction), the second question has been completely answered by A. Brown, F. Rodriguez Hertz, and Z. Wang \cite{BrownRodriguez-HertzWang2017}. In \cite{BrownRodriguez-HertzWang2017}, the authors show that smooth Anosov actions on (infra-)nilmanifolds, by higher rank lattices, are smoothly conjugated to actions by affine maps, provided that the action lifts to the universal cover. The proofs in \cite{BrownRodriguez-HertzWang2017} rely in an essential way on the fact that the action contains an Anosov element, since the Anosov element is topologically conjugated to an affine map \cites{Franks1969,Manning1974}. This conjugacy is a starting point for producing a conjugacy for the lattice action. In contrast, a partially hyperbolic diffeomorphism on a nilmanifold is not, in general, topologically conjugated to an affine map. The main novelty of this paper is producing a topological conjugacy from a $\Gamma-$action that contains a partially hyperbolic element to an affine action. This argument uses, in an essential way, that the action is of higher rank and fails in rank $1$ (in fact, on the manifolds considered in Theorems \ref{MainThm:AbelianCase} and \ref{MainThm:NonAbelianCase}, generic partially hyperbolic diffeomorphisms are not conjugated to affine maps). Preceding \cite{BrownRodriguez-HertzWang2017}, local rigidity for Anosov lattice actions on nilmanifolds was established in \cite{KatokSpatzier1997} (see also \cites{KatokLewis1991,Hurder1992}). Later Margulis and Qian \cite{Margulis--Qian} showed global topological rigidity for higher rank Anosov lattice actions on nilmanifolds under the additional assumption that the action preserves a fully supported measure. In fact, in \cite{Margulis--Qian}, they use a weaker hyperbolicity assumption, called {\bf weak hyperbolicity}, so the action as a whole does not need to contain an Anosov diffeomorphism. However, they still assume that there are enough partially hyperbolic elements so that the full tangent space is generated by the sum of many stable distributions of partially hyperbolic elements of the action.

A key tool when studying Anosov $\Gamma-$actions has been the theory of smooth $\Zbb^{k}$ or $\Rbb^{k}-$actions. To get a smooth conjugacy in \cite{BrownRodriguez-HertzWang2017}, the authors use previous rigidity results on \textbf{higher rank $\mathbb{Z}^{k}-$actions}\footnote{See Definition \ref{Def:HigherRankAbelianAction}} \cite{Rodriguez-HertzWang2014}, which showed that a higher rank Anosov $\Zbb^{k}-$ action on a nilmanifold is smoothly conjugated to an affine action. More generally, one of the main reasons we expect rigidity phenomenons to occur for higher rank lattice actions is that higher rank lattices contain a higher rank free abelian group. There is a parallel research program studying rigidity properties of higher rank abelian actions. A main question in this area is the \textbf{Katok--Spatzier conjecture}, see \cite{KatokSpatzier1994}. The Katok--Spatzier conjecture states that: if $\alpha:\Zbb^{k}\to \Diff(M)$ is a smooth action on a closed manifold, $\alpha$ have no rank$-1$ factor\footnote{roughly, the action has a rank$-1$ factor if it factors through a $\Zbb$-action on some quotient, see Section \ref{SubSec:HigherRankActions}.}, and there is one element $\mathbf{n}\in \Zbb^{k}$ such that $\alpha(\mathbf{n})$ is an Anosov diffeomorphism, then $M$ is diffeomorphic to a (infra-)nilmanifold and $\alpha$ is smoothly conjugated to an affine action. Again, this conjecture can be decomposed into two different parts, 1) providing a homogeneous structure on the manifold, and 2) classifying all Anosov $\Zbb^{k}-$actions on (infra-)nilmanifolds provided that there is no rank$-1$ factor. In \cite{Rodriguez-HertzWang2014} F. Rodriguez Hertz and Z. Wang answered the second question by proving that all such actions are (up to a $C^{\infty}$ coordinate change) by affine maps.

Leaving the Anosov setting, if we only assume that there exists a partially hyperbolic diffeomorphism, much less is known concerning the classification of $\Gamma-$actions. Note, however, that there are many algebraic models of $\Gamma-$actions that do not contain Anosov diffeomorphisms, including the models considered in this paper (see Theorems \ref{MainThm:AbelianCase} and \ref{MainThm:NonAbelianCase}). One direction in which many results have been produced is \textbf{local rigidity} of higher rank actions with a partially hyperbolic element. That is, we fix a model action $\alpha_{0}$ and study actions in a small neighborhood around $\alpha_{0}$. For abelian higher rank actions on nilmanifolds, the local question was initially solved on tori and is now solved on all nilmanifolds \cites{Wang2022, DamjanovicKatok2010}. The proofs in both \cites{Wang2022,DamjanovicKatok2010} use a \textbf{KAM scheme} to produce a conjugacy, this is an inherently local method that does not generalize to the non-local setting.  In \cite{FisherMargulis2009} D. Fisher and G. Margulis proved that all affine $\Gamma-$actions on homogeneous spaces are locally rigid, completely resolving the question of local rigidity around affine model systems. Previously, local rigidity of partially hyperbolic lattice actions on tori was known \cite{NiticaTorok}. As mentioned above, in \cite{Margulis--Qian} Margulis and Qian showed a global rigidity result, but only for actions that have enough partially hyperbolic elements so that the sum of stable distributions of many elements covers the entire tangent space.  

Recently, the second author proved global rigidity of partially hyperbolic higher rank $\mathbb{Z}^{k}-$actions on Heisenberg nilmanifolds, provided that the partially hyperbolic element has $1-$dimensional center (and satisfy a technical hypothesis on the stable and unstable leaves) \cite{Sandfeldt2024}. Also, the first author classified partially hyperbolic $\Gamma-$actions on manifolds of certain dimensions, under the assumption that the action has an invariant volume \cite{HL:dom}. However, in general, global rigidity of partially hyperbolic $\Gamma-$actions on nilmanifolds is not well understood. 

In this paper, we investigate higher rank lattice actions with a partially hyperbolic diffeomorphism. Roughly, we give an affirmative answer to the second part of Question \ref{Question:MainQuestion} on certain nilmanifolds when the $\Gamma-$action contains a partially hyperbolic diffeomorphism with $1-$dimensional center. For precise statements see Theorems \ref{MainThm:AbelianCase} and \ref{MainThm:NonAbelianCase}. The main novelty of this paper is that we do not assume the existence of an Anosov element, so there is, a priori, no single diffeomorphism in the action that is topologically conjugated to an affine map (as opposed to \cites{BrownRodriguez-HertzWang2017,KatokLewis1991,Hurder1992}, where results of Franks and Manning \cites{Franks1969,Manning1974} guarantee that the Anosov diffeomorphisms in the action are conjugated to affine maps). Moreover, the algebraic models considered do not support Anosov diffeomorphisms, so our results are the first global rigidity results on these manifolds.

Throughout the paper, we will use the following notation:
\begin{notation}\label{nota1}
    Let $G$ be a connected semisimple Lie group without compact factors, and with finite center. Moreover, assume that every simple factor of $G$ has $\mathbb{R}-$rank at least $2$. Let $\Gamma$ be a lattice in $G$.
\end{notation}

% \cite{Wang2022},\cite{KatokLewis1991},\cite{DamjanovicKatok2010},\cite{KatokSpatzier1994}
%   Recently, major advance in this direction came from Brown, F. Rodriguez-Hertz and Z. Wang in \cite{BrownRodriguez-HertzWang2017} where they prove that any smooth $\Gamma-$action on a compact nilmanifold containing an Anosov diffeomorphism is smoothly conjugated to an algebraic action. Before \cite{BrownRodriguez-HertzWang2017}, 

% Nevertheless, without assuming the existence of Anosov diffeomorphisms, some rigidity results were known under further assumptions, for instance, \cite{FisherMargulis2009},\cite{HL:Dom}. 

% %For instance, the local rigidity of any affine $\Gamma$ actions is obtained by Margulis--Fisher, \cite{Margulis--Fisher}. Under the dimension assumption, the classifications on manifolds by a higher rank lattice with a partially hyperbolic diffeomorphism in certain dimension by the first author's work \cite{HL:Dom}.

\subsection{Main theorems}
To state our main theorems, we need a notion of partially hyperbolic diffeomorphisms with well-behaved stable and unstable foliations. We say that a diffeomorphism $f:X\to X$ is \textbf{partially hyperbolic} if the tangent bundle splits into $Df-$invariant bundles $TX = E^{s}\oplus E^{c}\oplus E^{u}$ where $E^{s}$ is contracted, $E^{u}$ is expanded, and the behaviour along $E^{c}$ is dominated by the behaviour along $E^{s}$ and $E^{u}$ (for a precise definition see Section \ref{SubSec:FiberedPH}). We say that $f$ is \emph{${\rm QI}-$partially hyperbolic} if $f$ is partially hyperbolic and the stable and unstable foliations of $f$ lift to foliations with quasi-isometric leaves in the universal cover $\Tilde{X}$ (see Definition \ref{Def:QuasiIsometricLeaves}).

Our first main theorem is about higher rank lattice actions on $2-$step nilmanifolds with $1-$dimensional derived subgroup, that contain a QI-partially hyperbolic element. Any affine map on such a nilmanifold has a neutral direction (along the derived subgroup), so any action by affine maps is either parabolic or partially hyperbolic, but not Anosov. We show that, conversely, if an action by a higher rank lattice contains a QI-partially hyperbolic diffeomorphism then the action is (after a smooth coordinate change) by affine maps.
\begin{theorem}\label{MainThm:NonAbelianCase}
    Let $G$ and $\Gamma$ be as in \Cref{nota1}. Let $X_{\Lambda} = \Lambda\setminus N$ be a $2-$step nilmanifold with $\dim[N,N] = 1$. Let $\alpha\colon \Gamma\to{\rm Diff}^{\infty}(X_{\Lambda})$ be a smooth $\Gamma$ action on $X_{\Lambda}$. Assume that 
    \begin{enumerate}[label = (\roman*)]
    \item there is an element $\gamma_{0}\in\Gamma$ such that $\alpha(\gamma_{0})$ is ${\rm QI-}$partially hyperbolic with $1-$dimensional center and
	\item $\alpha$ lifts to an action on the universal cover $N$.
    \end{enumerate}
   Then there is a homomorphism $\Hat{\rho}:\Gamma\to{\rm Aff}(X_{\Lambda})$ and a diffeomorphism $H:X_{\Lambda}\to X_{\Lambda}$ such that
   \begin{align}
    H(\alpha(\gamma)x) = \Hat{\rho}(\gamma)H(x)
   \end{align}
   for all $\gamma\in\Gamma$. There is also a finite index subgroup $\Gamma'\leq\Gamma$ such that $\Hat{\rho}|_{\Gamma'}:\Gamma'\to{\rm Aut}(X_{\Lambda})$.
\end{theorem}
The existence of a topological conjugacy can also be guaranteed in finite regularity, see \Cref{sec:topconj}.

\begin{remark} In the statement, it is worth mentioning that the assumption about lifting can be achieved in many cases, for instance, when $\Gamma$ is cocompact. For more discussions on the lifting property, see \cite{BrownRodriguez-HertzWang2017}*{Remark 1.5 and Section 9}.
\end{remark}
Note that the manifold $X_{\Lambda}$ in Theorem \ref{MainThm:NonAbelianCase} is a fiber bundle over a torus $X_{\Lambda}\to X_{\Lambda}/[N,N]\cong\mathbb{T}^{d}$. An action $\Hat{\rho}:\Gamma\to{\rm Aut}(X_{\Lambda})$ by automorphisms preserves this fiber bundle structure since the fibers coincide with the orbits of $[N,N]$, which are in turn preserved by automorphisms. We obtain a result similar to Theorem \ref{MainThm:NonAbelianCase} for the trivial fiber bundle $\mathbb{T}^{d}\times\mathbb{T}\to\mathbb{T}^{d}$, considering model actions that respects this fiber bundle structure: $\Hat{\rho}:\Gamma\to{\rm GL}(d,\mathbb{Z})\times 1$.
\begin{theorem}\label{MainThm:AbelianCase}
Let $G$ and $\Gamma$ be as in \Cref{nota1}. Let $\mathbb{T}^{d+1}$ be the $(d+1)$-dimensional torus and let $\alpha:\Gamma\to{\rm Diff}^{\infty}(\mathbb{T}^{d+1})$ be a smooth action. Assume that
\begin{enumerate}[label = (\roman*)]
    \item there is an element $\gamma_{0}\in\Gamma$ such that $\alpha(\gamma_{0})$ is ${\rm QI}$-partially hyperbolic with $1-$dimensional center,
    \item $\alpha$ lifts to an action on the universal cover,
    \item the induced map on fundamental group $\alpha_{*}:\Gamma\to{\rm GL}(d+1,\mathbb{Z})$ takes values in ${\rm GL}(d,\mathbb{Z})\times 1$.
\end{enumerate}
Then there are homomorphisms $\rho:\Gamma\to{\rm GL}(d,\mathbb{Z})$, $\psi:\Gamma\to\mathbb{T}$, and a diffeomorphism $H:\mathbb{T}^{d+1}\to\mathbb{T}^{d+1}$ such that $H\alpha(\gamma)H^{-1}(x,t) = (\rho(\gamma)x,t + \psi(\gamma))$. That is, $\alpha$ is conjugated to an affine action.
\end{theorem}
%Note that any homomorphism $\psi:\Gamma\to\mathbb{T}$, with $\Gamma$ as in the theorem, has a finite image so up to finite index, $\alpha$ is smoothly conjugated to $\rho\times 1$.
\begin{remark} \Cref{MainThm:AbelianCase} generalizes the main results of \cite{NiticaTorok} in two important ways: 1) it is non-local and 2) we consider any higher rank lattice (as opposed to only considering ${\rm SL}(d,\mathbb{Z})$). 
\end{remark}
\begin{remark}
In both Theorems \ref{MainThm:AbelianCase} and \ref{MainThm:NonAbelianCase} we assume that there is a QI-partially hyperbolic $f$ in the action $\alpha$. In Appendix \ref{Appendix:QI-Condition} (Lemma \ref{L:SufficientConditionForQI}) we show that the QI-condition is implied if, for example, $f$ is conjugated (or leaf conjugated) to an affine map on $X_{\Lambda}$. In particular, if $\alpha:\Gamma\to{\rm Diff}^{\infty}(X_{\Lambda})$ is an action that satisfies the conclusion of either Theorem \ref{MainThm:NonAbelianCase} or Theorem \ref{MainThm:AbelianCase} then $\alpha$ also satisfy the assumptions of the corresponding theorem, so the assumptions are sharp in this sense. Motivated by \cites{Hammerlindl2013,HammerlindlPotrie2015}, it seems likely that any partially hyperbolic diffeomorphism with $1-$dimensional center on a manifold as in Theorem \ref{MainThm:NonAbelianCase} or \ref{MainThm:AbelianCase} has quasi-isometric leaves. If this is true, then QI can be removed as an assumption in both Theorems.
\end{remark}

%Especially, combining with Theorem \ref{thm:svenHeis}, we can show global rigidity of  $\Gamma$ action on Heisenberg nilmanifolds assuming existence of one QI-partially hyperbolic diffeomorphism with one dimensional center. 
%\begin{corollary}\label{coro:heis}
%Let $G$ be a higher rank semisimple Lie group without compact factor with finite center. Assume that every simple factor of $G$ has real rank at least $2$. Let $\Gamma$ be an irreducible lattice in $G$. Let $M$ be a Heisenberg nilmanifold. Let $\alpha\colon \Gamma\to \textrm{Diff}^{r}(M)$ a smooth $\Gamma$ action on $M$. \footnote{$r$ can be 1+hol} Assume that there is an element $\gamma_{0}\in\Gamma$ such that $\alpha(\gamma_{0})$ is a QI-partially hyperbolic diffeomorphism with one dimensional center. Then there is a $\Gamma$ invariant absolutely continuous measure $\mu$ on $M$. Moreover, $\alpha$ is topologically conjugate to an affine action. If $\alpha$ is $C^{\infty}$ action then $\alpha$ is $C^{\infty}$ conjugate to an affine action. 
%\end{corollary}

\subsection*{Acknowledge}
H.\ L.\ was supported by an AMS-Simons Travel Grant.  The authors also benefitted from hospitality of Institut Henri Poincar\'e while working on this project (UAR 839 CNRS-Sorbonne Universit\'e, ANR-10-LABX-59-01). S. S. is grateful for the hospitality of Northwestern University, where part of this paper was written. The authors thank David Fisher for useful comments on an earlier version of this paper. The authors also thank Amie Wilkinson and Aaron Brown for useful discussions.

\section{Preliminaries and reductions}
%\subsection{Cocycles and twisted cocycles}
\subsection{(Measurable) Cocycles}
%Throughout this section, let $D$ be a discrete finitely generated group and $(X,\mu)$ be a ergodic $D$ space. Assume that $\mu$ is $D$-invariant. We can trivialize cocycles for certain class of groups. 
We use several cocycle rigidity results to prove the main theorems. The first cocycle rigidity result we will use is when the acting group $D$ has Kazhdan's property (T). %\textcolor{red}{[Need version of this theorem for circle, not just $\mathbb{R}$.]}

\begin{theorem}[{\cite{MR777345}*{Chapter 9}}]\label{Thm:MeasurableTrivializationOfCocycle}
    Let $D$ be a discrete group with Kazhdan's property (T). Let $(X,\mu)$ be an ergodic $D$-space. Assume that $\mu$ is $D$-invariant and let $\beta:D\times X\to \Rbb^{k}$ be a measurable cocycle. Then there exists a measurable map $\phi:X\to \mathbb{R}$ such that \[\beta(\gamma,x)=\phi(\gamma.x)-\phi(x)\] for all $\gamma\in D$ and $\mu$ almost every $x\in X$.
\end{theorem}

%Let $\alpha$ be a $C^{\infty}$ action on a nilmanifold $X$ that lifts on a universal cover $N$. 

%Let $\rho$ be the associated linear data. We assume that there is a homeomorphism  $H:X\to X$ so that $H\circ\alpha(\gamma)=\rho(\gamma)\circ H$ for all $\gamma\in \Gamma$.
%Let $\mathbb{G}$ be an algebraically simply connected, connected, semisimple algebraic group defined over $\mathbb{R}$. Let $G=\mathbb{G}(\mathbb{R})$ be a real points of $\mathbb{G}$. Assume that each simple factors have real rank at least $2$. Let $\Gamma$ be a lattice in $G$.  %$D=\Gamma$, we can solve cocycle problems for all cocycles over ergodic $\Gamma$ actions.

Let $(X,\mu)$ be $D$-ergodic space. Assume that $\mu$ is $D$-invariant. Recall that when we have a cocycle $\beta:D\times X \to H$ and a $H-$space $Y$ then we can define a skew product action on $X\times_{\beta}Y$ by $\gamma.(x,y)=(\gamma.x,\beta(\gamma,x).y)$ for $\gamma\in D$. If the $H$ action preserves a probability measure $\nu$ then the measure $\mu\otimes \nu$ on $X\times_{\beta}Y$ is $D-$invariant. 
\begin{example}[Suspension] Let $G$ and $\Gamma$ be as in \Cref{nota1}. Let $\Rcal:G\times G/\Gamma\to \Gamma$ be a return cocycle. That is, after fixing a fundamental domain $Y$ for $\Gamma$ in $G$ and identifying $Y$ with $G/\Gamma$ the cocycle $\Rcal$ is defined by $\Rcal(g,y)=\gamma$ if and only if $gy\gamma^{-1}\in Y$. Since $G/\Gamma$ is an ergodic $G$-space, if $(X,\mu)$ is a $\Gamma-$ergodic space, then we can define a suspension $G/\Gamma\times_{\Rcal} X$ with a $G$ action. More precisely, the action is given by $g.(y,x)=(gy\Rcal(g,y)^{-1}, \Rcal(g,y).x)$. We call the $\Gamma-$action on $(X,\mu)$ \emph{induced-irreducible} if the $G-$action on the suspension is irreducible (i.e. each simple factor of $G$ acts ergodically).
\end{example}
For ${\rm Diff}(\mathbb{T})-$valued cocycles the following theorem states that we can make the cocycle cohomologous to a $\textrm{Isom}(S^{1})-$valued cocycle.
\begin{theorem}[D.Witte Morris--Zimmer, \cite{MorrisZimmer}]\label{Thm:SolvingDiffTValuedCocycle}
    Let $\Gamma$ be as in \Cref{nota1}. Assume that $(X,\mu)$ is an induced irreducible $\Gamma-$space and $\mu$ is $\Gamma-$invariant. Let $\beta:\Gamma\times X\to{\rm Diff}^{1}(\mathbb{T})$ be a measurable cocycle and $X\times_{\beta}\mathbb{T}$ the skew extension. Then $\mu\otimes\lambda$ is a $\Gamma-$invariant measure on $X\times\mathbb{T}$ where $\lambda$ is the Lebesgue measure on $\mathbb{T}$. Hence, $\beta$ is measurably cohomologous to $\textrm{Isom}(\mathbb{T})$ (as a cocycle into $\textrm{Homeo}(\mathbb{T})$).
    %Suppose that for all $\gamma\in \Gamma$, the map $x\mapsto \beta(\gamma,x)$ is in the bounded subset of $\textrm{Diff}^{2}(S^{1})$. Then there is a measurable map $\phi:X\to \textrm{Homeo}_{+}(S^{1})$ such that 
   % \[\phi(\gamma.x)^{-1}\beta(\gamma,x)\phi(x)\in \textrm{Rot}(S^{1})\] where $\textrm{Rot}(S^{1})$ is the group of rotations on $S^{1}$.
\end{theorem}
Indeed, when the target group is $\textrm{Diff}^{r}(\mathbb{T})$ with $r\ge 3/2$ then one can deduce the above theorem from \cite{Navas} only using the fact that $\Gamma$ has Kazhdan's property (T).

Note that when $(X,\mu)$ is a torus with $\mu$ the Haar measure and the $\Gamma-$action is given by toral automorphisms, then the $\Gamma-$action is induced irreducible by \cite{HL:DCSR}*{Theorem 5.6}. This is the action we apply \cref{Thm:SolvingDiffTValuedCocycle} to later (\cref{lem:cocycleproblem}).

\subsection{Compact nilmanifolds}

Let $N$ be a (connected) Lie group with Lie algebra $\mathfrak{n}$. Let $\mathfrak{n}_{(j+1)} = [\mathfrak{n},\mathfrak{n}_{(j)}]$, $\mathfrak{n}_{(1)} = \mathfrak{n}$ and $N_{(j)}$ the associated subgroup. We say that $N$ (and $\mathfrak{n}$) is $\ell-$step nilpotent if $\mathfrak{n}_{(\ell)}\neq0$ and $\mathfrak{n}_{(\ell+1)} = 0$. Note that $N_{(j)}$ is normal so we obtain projection maps $\pi_{(j)}:N\to N/N_{(j)}$. We say that a subgroup $\Lambda\leq N$ is a \textbf{lattice} if $\Lambda\setminus N$ has finite volume (or equivalently, in the case of nilpotent groups \cite{CorwinGreenleaf1990}, $\Lambda\setminus N$ is compact). Given a nilpotent lie group $N$ and a lattice $\Lambda$ we will denote the associated \textbf{compact nilmanifold} by $X_{\Lambda} = \Lambda\setminus N$.
\begin{example}
Let $d$ be an even integer and $\omega$ the standard symplectic structure on $\mathbb{R}^{d}$. Let $N = \mathbb{R}^{d}\times\mathbb{R}$ and define a group operation:
\begin{align}
(\mathbf{x},t)\cdot(\mathbf{y},s) = \left(\mathbf{x}+\mathbf{y},t+s + \frac{\omega(\mathbf{x},\mathbf{y})}{2}\right).
\end{align}
The group $N$ is a $2-$step nilpotent Lie group which we call the $(d+1)-$dimensional \textbf{Heisenberg group}. We obtain a lattice in $N$ as $\Lambda = \mathbb{Z}^{d}\times(\mathbb{Z}/2)$. Given any lattice $\Lambda\leq N$ we say that $X_{\Lambda} = \Lambda\setminus N$ is a \textbf{Heisenberg nilmanifold}. It is immediate that $[N,N] = 0\times\mathbb{R}$, so $\dim[N,N] = 1$. That is, $N$ is $2-$step nilpotent with $1-$dimensional derived subgroup. More generally, every $2-$step nilpotent Lie group with $1-$dimensional derived subgroup is a product of a Heisenberg group and a Euclidean space. 
\end{example}
When $N$ is $2-$step nilpotent we will denote the map $\pi_{(2)}:N\to N/[N,N]$ by 
\begin{align}
\pi:N\to N/[N,N].
\end{align}
The map $\pi$ then descends to a map $\pi:X_{\Lambda}\to X_{\Lambda}/[N,N]\cong\mathbb{T}^{d}$ \cite{CorwinGreenleaf1990}.

We denote by ${\rm Aut}(N)$ the automorphisms of $N$. If $\Lambda\leq N$ is a lattice then we write
\begin{align}
{\rm Aut}(X_{\Lambda}) = \{A\in{\rm Aut}(N)\text{ : }A\Lambda = \Lambda\}.
\end{align}
That is, ${\rm Aut}(X_{\Lambda})$ are precisely the automorphisms that descends to maps $X_{\Lambda}\to X_{\Lambda}$. Note that any $A\in{\rm Aut}(X_{\Lambda})$ preserve $[N,N]$, so there is a map 
\begin{align}
q:{\rm Aut}(X_{\Lambda})\to{\rm Aut}(X_{\Lambda}/[N,N])\cong{\rm Aut}(\mathbb{T}^{d})\cong{\rm GL}(d,\mathbb{Z})
\end{align}
such that $q(A)\pi = \pi A$.

\subsection{Fibered partially hyperbolic diffeomorphisms}\label{SubSec:FiberedPH}

Let $X$ be a smooth, closed manifold and $f:X\to X$ a diffeomorphism. We say that $f$ is (absolutely) \textbf{partially hyperbolic} if there is a splitting $TX = E^{s}\oplus E^{c}\oplus E^{u}$ and constants $C\geq1$, $\lambda,\Hat{\lambda},\mu,\Hat{\mu}\in(0,1)$ such that $\lambda < \Hat{\lambda}$, $\mu < \Hat{\mu}$ and
\begin{align}
& \norm{D_{x}f^{n}(v^{s})}\leq C\lambda^{n}\norm{v^{s}},\quad v^{s}\in E^{s} \\
& \frac{1}{C}\Hat{\lambda}^{n}\norm{v^{c}}\leq\norm{D_{x}f^{n}(v^{c})}\leq C\Hat{\mu}^{-n}\lambda^{n}\norm{v^{s}},\quad v^{s}\in E^{s}, \\
& \frac{1}{C}\mu^{-n}\norm{v^{u}}\leq\norm{D_{x}f^{n}(v^{u})}\quad v^{u}\in E^{u}.
\end{align}
That is, $f$ is partially hyperbolic if $E^{s}$ is contracted, $E^{u}$ is expanded, and the behaviour of $E^{c}$ is dominated by the behaviour of $f$ along $E^{s}$ and $E^{u}$. It is well-known that $E^{s}$ and $E^{u}$ integrate uniquely to continuous foliations with smooth leaves, denoted $W^{s}$ and $W^{u}$.
\begin{example}
If $L\in{\rm GL}(d,\mathbb{Z})\cong{\rm Aut}(\mathbb{T}^{d})$ (or more generally if $L\in{\rm Aut}(X_{\Lambda})$ for some compact nilmanifold) then the induced map $L:\mathbb{T}^{d}\to\mathbb{T}^{d}$ (or $L:X_{\Lambda}\to X_{\Lambda}$) is partially hyperbolic if $L$ (or $DL:\mathfrak{n}\to\mathfrak{n}$) has eigenvalues with modulus distinct from $1$. In this case, we can take $E^{s}$ as all eigendirections of eigenvalues with modulues $<1$, $E^{u}$ all eigendirections of eigenvalues with modulus $> 1$ and $E^{c}$ all eigendirections of eigenvalues with modulus $=1$.
\end{example}
Let $\Tilde{X}\to X$ be the universal cover of $X$. The foliations $W^{s}$ and $W^{u}$ lift to foliations $\Tilde{W}^{s}$ and $\Tilde{W}^{u}$ in $\Tilde{X}$. We will denote by $\intd_{\sigma}$, $\sigma = s,u$, the distance between points $x,y\in W^{\sigma}(z)$ along the leaf $W^{\sigma}(z)$.
\begin{definition}\label{Def:QuasiIsometricLeaves}
Let $f:X\to X$ be partially hyperbolic. We say that $f$ is \textbf{QI-partially hyperbolic} if there is a constant $Q$ such that for all $x,y\in\Tilde{W}^{\sigma}(z)$, $\sigma = s,u$, we have
\begin{align}
\intd(x,y)\leq\intd_{\sigma}(x,y)\leq Q\intd(x,y).
\end{align}
That is, the distance function $\intd_{\sigma}$ along $\Tilde{W}^{\sigma}$ is comparable to the distance function $\intd$ in $\tilde{X}$.
\end{definition}
\begin{definition}
Let $f:X\to X$ be partially hyperbolic. We say that $f$ is \textbf{fibered partially hyperbolic} if there is a fiber bundle $p:X\to B$ and an Anosov homeomorphism $\Hat{f}:B\to B$ such that $p\circ f = \Hat{f}\circ p$ and the fibers of $p$ are smooth manifolds whose associated distribution is the center distribution $E^{c}$ (see \cite{AvilaVianaWilkinson2022}).
\end{definition}
\begin{example}
Let $X_{\Lambda}$ be a Heisenberg nilmanifold of dimension $d+1$. Any $A\in{\rm Sp}(d,\mathbb{Z})$ defines an element $L\in{\rm Aut}(X_{\Lambda})$ by $L:(\mathbf{x},t)\mapsto(A\mathbf{x},t)$. If $A$ is a hyperbolic matrix then the projection $\pi:X_{\Lambda}\to\mathbb{T}^{d}$ makes $L$ a fibered partially hyperbolic diffeomorphism. More generally, it is shown in \cite{Sandfeldt2024} that any QI-partially hyperbolic diffeomorphism on $X_{\Lambda}$ with $1-$dimensional center is fibered. In fact, in \cite{Hammerlindl2013} (see also \cite{HammerlindlPotrie2015}) it is shown that in dimension $3$ \textbf{every} partially hyperbolic diffeomorphism on $X_{\Lambda}$ is fibered.
\end{example}
\begin{example}
Let $f:X\to X$ be Anosov and $g:M\to M$ an isometry. The product map $f\times g:X\times M\to X\times M$ is then fibered with the trivial fiber bundle $X\times M\to X$. A special case of this construction, of importance in this paper, is the case when $f = A\in{\rm GL}(d,\mathbb{Z})$ is a hyperbolic toral automorphism and $M = \mathbb{T}$ is a circle.
\end{example}

\subsection{Linear abelian higher rank action on tori}\label{SubSec:HigherRankActions}

Let $\Gamma$ be a discrete group,  $\rho:\Gamma\to{\rm GL}(d,\mathbb{Z})$ a homomorphism, and $\mathbb{Z}^{k}\cong\Sigma\subset\Gamma$ be a subgroup. We say that $\chi:\Sigma\to\mathbb{R}$ is a \textbf{Lyapunov exponent} (or \textbf{Lyapunov functional}) of $\rho|_{\Sigma}$ if there is a vector $\mathbf{v}\neq 0$ such that
\begin{align}
\lim_{n\to\infty}\frac{1}{n}\log\norm{\rho(na)\mathbf{v}} = \chi(a),\quad a\in\Sigma.
\end{align}
Given a Lyapunov exponent $\chi$ we define the associated \textbf{coarse Lyapunov exponent} $[\chi]$ of $\rho|_{\Sigma}$ to be the equivalence class of all exponents $\chi'$ of $\rho|_{\Sigma}$ such that $\chi' = c\chi$ for some constant $c > 0$.
\begin{definition}
For a homomorphism $\rho:\Gamma\to{\rm GL}(d,\mathbb{Z})$ and an abelian subgroup $\Sigma\leq\Gamma$ we define $\Delta_{\Sigma}$ as the collection of all coarse exponents of $\rho|_{\Sigma}$.
\end{definition}
Given a Lyapunov exponent $\chi$ of $\rho|_{\Sigma}$ we define
\begin{align}
E_{\Sigma}^{\chi} := \left\{\mathbf{v}\neq0\text{ : }\lim_{n\to\infty}\frac{1}{n}\log\norm{\rho(na)\mathbf{v}} = \chi(a),\text{ for all }a\in\Sigma\right\}\cup\{0\},
\end{align}
it is immediate that $E_{\Sigma}^{\chi}$ is a subspace. We also define the \textbf{coarse subspace} associated to $[\chi]$ as
\begin{align}
E_{\Sigma}^{[\chi]} = \bigoplus_{\chi'\in[\chi]}E_{\Sigma}^{\chi'}.
\end{align}
Since the image of $\rho|_{\Sigma}$ can be conjugated into simultaneous Jordan form, it is immediate that we have
\begin{align}
\mathbb{R}^{d} = \bigoplus_{[\chi]\in\Delta_{\Sigma}}E_{\Sigma}^{[\chi]}.
\end{align}
That is, we have a decomposition of $\mathbb{R}^{d}$ into coarse subspaces.
\begin{definition}
We say that two coarse exponents $[\chi]$ and $[\eta]$ are dependent if for any two representatives $\chi'$, $\eta'$ there is a constant $c > 0$ such that $\eta' = -c\chi'$. If $[\chi]$ and $[\eta]$ are not dependent then we say that they are independent.
\end{definition}
We will give a definition of higher rank abelian actions formulated in a useful way for our applications. However, the definition is equivalent to the standard definition (see \cite{Sandfeldt2024}*{Lemma 2.2}).
\begin{definition}\label{Def:HigherRankAbelianAction}
We say that $\rho|_{\Sigma}$ is higher rank if for any $[\chi]\in\Delta_{\Sigma}$ the space
\begin{align}
V = \bigoplus_{[\lambda]\neq\pm[\chi]}E_{\Sigma}^{[\lambda]}
\end{align}
defines a minimal translation action on $\mathbb{T}^{d}$.
\end{definition}

\subsection{Reduction to effective actions of lattices in algebraic groups}\label{SubSec:InitialReductions}

Let $\Gamma$ be as in Theorems \ref{MainThm:NonAbelianCase} and \ref{MainThm:AbelianCase}, $X_{\Lambda}$ as in either Theorem \ref{MainThm:NonAbelianCase} or Theorem \ref{MainThm:AbelianCase}, and $\alpha:\Gamma\times X_{\Lambda}\to X_{\Lambda}$ an action satisfying either $(i)$, $(ii)$ from Theorem \ref{MainThm:NonAbelianCase} or $(i)$, $(ii)$, $(iii)$ from Theorem \ref{MainThm:AbelianCase}. Let $\gamma_{0}\in\Gamma$ be the element such that $f = \alpha(\gamma_{0})$ is partially hyperbolic. In particular, for any $n > 0$ we have $\alpha(\gamma_{0}^{n})\neq{\rm id}$ so $\alpha$ has infinite image. It follows that $\ker\alpha\subset\Gamma$ is a normal subgroup that does not have finite index in $\Gamma$. Note that $\Gamma/\ker \alpha$ is still a higher rank lattice, hence, we have proved the following lemma.
\begin{lemma}\label{L:ActionIsEffective}
We may assume without loss of generality that the action $\alpha$ is effective.
\end{lemma}
In \cite{Sandfeldt2024} a fiber bundle structure $\Phi:X_{\Lambda}\to\mathbb{T}^{d}$ that conjugates $f$ to a hyperbolic automorphism is produced. More precisely, the following holds \cite{Sandfeldt2024}*{Theorem 1.1}.
\begin{theorem}\label{Thm:PropertiesOfFibration}
The following properties hold for $f$:
\begin{enumerate}[label = (\roman*)]
    \item $f$ is dynamically coherent with global product structure,
    \item all foliations $W^{\sigma}$, $\sigma = s,c,u,cs,cu$, are uniquely integrable,
    \item the center foliation $W^{c}$ have compact oriented circle leaves,
    \item there is a Hölder $\Phi:X_{\Lambda}\to\mathbb{T}^{d}$ such that $\Phi^{-1}(\Phi(x)) = W^{c}(x)$, $\Phi$ is homotopic to the projection $\pi$ and $\Phi(fx) = L_{su}\Phi(x)$ where $L_{su}\in{\rm GL}(d,\mathbb{Z})$ is hyperbolic,
\end{enumerate}
moreover, if $N$ is not abelian then
\begin{enumerate}[label = (\roman*)]
    \setcounter{enumi}{4}
    \item $f$ is accessible.
\end{enumerate}
\end{theorem}
We write $\alpha_{*}:\Gamma\to{\rm Aut}(\pi_{1}X_{\Lambda}) = {\rm Aut}(X_{\Lambda})$ for the induced map on fundamental group. Recall that we have a projection map $\pi:X_{\Lambda}\to\mathbb{T}^{d}$ defined by $x\mapsto x[N,N]$ when $N$ is non-abelian and projection onto the first $d$ coordinates when $N$ is abelian. Define a representation $\rho$ of $\Gamma$ by
\begin{align}
\rho:\Gamma\to{\rm GL}(d,\mathbb{Z}),\quad\pi_{*}\alpha_{*}(\gamma) =: \rho(\gamma)\pi_{*}.
\end{align}
Since $\rho(\gamma_{0}) = L_{su}$ is hyperbolic the map $\rho$ has infinite image. Combining Theorem \ref{Thm:PropertiesOfFibration} with \cite{BrownRodriguez-HertzWang2017}*{Theorem 3.2} we obtain a semiconjugacy between $\alpha$ and $\rho$.
\begin{theorem}\label{Thm:LatticeEquivariantFibration}
After possibly dropping to a finite index subgroup $\Gamma'$, the map $\Phi$ from Theorem \ref{Thm:PropertiesOfFibration} is $\Gamma'-$equivariant. That is, for any $\gamma\in\Gamma'$
\begin{align}
\Phi(\alpha(\gamma)x) = \rho(\gamma)\Phi(x)
\end{align}
where $\rho:\Gamma'\to{\rm GL}(d,\mathbb{Z})$. In particular, for any $\gamma\in\Gamma'$ we have $\alpha(\gamma)W^{c}(x) = W^{c}(\alpha(\gamma)x)$.
\end{theorem}
\begin{proof}
The proof of the theorem follows precisely as the proof of Theorem $1.3$ in \cite{BrownRodriguez-HertzWang2017}*{Section 7}. Indeed, the image of the map $\rho:\Gamma\to{\rm GL}(d,\mathbb{R})$ contain a hyperbolic matrix (the image of $\gamma_{0}$) so the proof in \cite{BrownRodriguez-HertzWang2017}*{Section 7}, checking that the assumptions in \cite{BrownRodriguez-HertzWang2017}*{Theorem 3.2} holds, applies.
\end{proof}
We will assume in the remainder that we have dropped to a finite index subgroup such that $\Phi(\alpha(\gamma)x) = \rho(\gamma)\Phi(x)$.
\begin{lemma}\label{L:HomotopyRepIsInjective}
We may assume, without loss of generality, that $\rho:\Gamma\to{\rm GL}(d,\mathbb{Z})$ is injective.
\end{lemma}
\begin{proof}
Let $\gamma\in\ker\rho$. By Theorem \ref{Thm:LatticeEquivariantFibration} we have $\Phi(\alpha(\gamma)x) = \Phi(x)$ so $\alpha(\gamma)W^{c}(x) = W^{c}(x)$ for every $x\in X_{\Lambda}$. It follows that $\ker\rho$ act on $W^{c}(x)$. Since $\ker\rho$ is normal in $\Gamma$ it is either finite or a semisimple lattice. Either way, since $W^{c}(x)$ is a circle, the image $\alpha(\ker\rho)\subset{\rm Diff}(W^{c}(x)$ is finite \cite{DeroinHurtado}. Note that the cardinality $|\alpha(\ker\rho)\subset{\rm Diff}(W^{c}(x)|$ is locally constant in $x$, so therefore constant. It follows that there is a finite index subgroup $H\leq\ker\rho$ that acts trivially on every $W^{c}(x)$. Since $h\in H$ also satisfy $\Phi(\alpha(h)x) = \Phi(x)$ it follows that $H\subset\ker\alpha$. But the action $\alpha$ is effective, so $H = e$. Since the group $e$ has finite index in $\ker\rho$ it follows that $\ker\rho$ is finite. Every element $\rho(\gamma)$ fix $0$ so we obtain a map $\psi:\Gamma\to{\rm Diff}(\Phi^{-1}(0))$ by restricting $\alpha$. Since $\Gamma$ has finite image in ${\rm Diff}(\Phi^{-1}(0))$ \cite{DeroinHurtado} it follows that $\Gamma' = \ker\psi$ is a finite index subgroup of $\Gamma$. It is clear that $\Gamma'\cap\ker\rho = e$ so, after possibly dropping to the finite index subgroup $\Gamma'$, the lemma follows. 
\end{proof}

%\subsection{Reduction to an algebraic group and arithmetic lattice}
We will reduce the groups $G$ that we consider to be the real points of a simply connected semisimple algebraic group, $G = \mathbf{G}(\mathbb{R})^{\circ}$. We will also show that any action as in Theorem \ref{MainThm:NonAbelianCase} or as in Theorem \ref{MainThm:AbelianCase} contain a subaction $\mathbb{Z}^{k}\cong\Sigma\subset\Gamma$ that is higher rank (in the setting of Theorem \ref{MainThm:AbelianCase} the action is higher rank on the base of the fibration $\mathbb{T}^{d+1}\to\mathbb{T}^{d}$, or equivalently we find a $\mathbb{Z}^{k}\cong\Sigma\subset\Gamma$ such that the action $\alpha|_{\Sigma}$ has precisely one rank$-1$ factor). The conclusion of the following lemma is Hypothesis $8.1$ in \cite{BrownRodriguez-HertzWang2017}.
\begin{lemma}\label{L:RestrictToAlgebraicGroup}
We may assume, in addition to the assumptions in Theorem \ref{MainThm:AbelianCase} or \ref{MainThm:NonAbelianCase}, without loss of generality that there is a simply connected semisimple algebraic group $\mathbf{G}$ defined over $\mathbb{R}$ such that all its $\mathbb{R}-$simple factors have $\mathbb{R}-$rank at least $2$, $G = \mathbf{G}(\mathbb{R})^{\circ}$, and $\Gamma\leq G$ is a lattice.
\end{lemma}
\begin{proof}
This is outlined in \cite{BrownRodriguez-HertzWang2017}*{Section 8.1}. The argument in Section 8.1 of \cite{BrownRodriguez-HertzWang2017} works in our setting as well, the only aspect that has to be changed is that we can not use that our action $\alpha:\Gamma\to{\rm Diff}^{\infty}(X_{\Lambda})$ is conjugated to an action by automorphisms. However, we can use that the map $\rho:\Gamma\to{\rm GL}(d,\mathbb{Z})$ is injective (from Lemma \ref{L:HomotopyRepIsInjective}) which implies that we can lift the action $\alpha$ as in \cite{BrownRodriguez-HertzWang2017}.
\end{proof}
\begin{lemma}\label{L:ExistenceOfHigherRankAbelianGroup}
If $S\leq\Gamma$ is a Zariski dense semigroup such that every $\gamma\in S$ satisfy that $\rho(\gamma)\in{\rm GL}(d,\mathbb{Z})$ is hyperbolic, then there is $\gamma\in S$ and a finite index subgroup $\Sigma\subset Z_{\Gamma}(\gamma)$ such that $\rho|_{\Sigma}:\Sigma\to{\rm GL}(d,\mathbb{Z})$ is higher rank. In particular, there is a subgroup $\mathbb{Z}^{k}\cong\Sigma\leq\Gamma$ such that $\rho|_{\Sigma}$ is higher rank.
\end{lemma}
\begin{proof}
By using Lemma \ref{L:RestrictToAlgebraicGroup} (corresponding to Hypothesis $8.1$ in \cite{BrownRodriguez-HertzWang2017}) the first paragraph of the proof of \cite{BrownRodriguez-HertzWang2017}*{Proposition 8.3} implies that we may assume that \cite{BrownRodriguez-HertzWang2017}*{Hypothesis 8.14} holds (with $\rho$ in \cite{BrownRodriguez-HertzWang2017}*{Hypothesis 8.14} corresponding to our map $\rho:\Gamma\to{\rm GL}(d,\mathbb{Z})$). The first part now follows from \cite{BrownRodriguez-HertzWang2017}*{Proposition 8.15}. The last part of the lemma follows since there always exists a Zariski dense semigroup containing only hyperbolic automorphisms (we can use the construction in \cite{BrownRodriguez-HertzWang2017}*{Section 8.2}, see also Section \ref{Sec:SmoothConjugacyAccessible}).
\end{proof}

\section{Topological rigidity}\label{sec:topconj}

Let $\Gamma$, $X_{\Lambda}$, and $\alpha$ be as in Section \ref{SubSec:InitialReductions}. We will assume in the remainder of this section that we have dropped to a finite index subgroup such that Theorem \ref{Thm:LatticeEquivariantFibration} applies, so $\Phi(\alpha(\gamma)x) = \rho(\gamma)\Phi(x)$. Let $\gamma_{0}\in\Gamma$ be the element such that $f = \alpha(\gamma_{0})$ is partially hyperbolic. In this section we start the proofs of Theorems \ref{MainThm:AbelianCase} and \ref{MainThm:NonAbelianCase} by proving that $\alpha$ is topologically conjugated to an affine action. In Section \ref{SubSec:CenterTranslation} we prove that $\alpha$ commute with an action $\eta_{c}:\mathbb{T}\times X_{\Lambda}\to X_{\Lambda}$ which acts transitively and freely on each fiber of $\Phi$ (or equivalently on each center leaf of $f$). This implies, in particular, that $\alpha$ preserves a metric along the center. After we start the construction of a topological conjugacy to an affine action. In Section \ref{SubSec:ConjugacyForAbelianSubgroup} we produce a topological conjugacy for an abelian subgroup $\mathbb{Z}^{k}\cong\Sigma\leq\Gamma$ such that $\rho|_{\Sigma}$ is of higher rank. Once we have a conjugacy for $\Sigma$ we extend it (in Section \ref{SubSec:ExtensionOfTopologicalConjugacy}) to a conjugacy for all of $\Gamma$.

\subsection{Existence of center translation}\label{SubSec:CenterTranslation}

Recall that the map $f:X_{\Lambda}\to X_{\Lambda}$ can be considered as a cocycle over $\rho(\gamma_{0}) =: L_{su}:\mathbb{T}^{d}\to\mathbb{T}^{d}$ using the map $\Phi:X_{\Lambda}\to\mathbb{T}^{d}$ (in the sense of \cite{AvilaViana2010}*{Section 2.1}). Since $L_{su}$ is hyperbolic we can apply the invariance principle by Avila-Viana \cite{AvilaViana2010}*{Theorem D}.
\begin{lemma}\label{lem:cocycleproblem}
There is a $\Gamma-$invariant measure $\mu$ such that $\Phi_{*}\mu = {\rm vol}_{\mathbb{T}^{d}}$, $\mu$ has conditionals $\mu_{x}^{c}$ equivalent to Lebesgue, and the conditional measures $\mu_{x}^{c}$ vary continuously in $x$.
\end{lemma}
\begin{proof}
By (measurably) trivializing the fiber bundle $\Phi:X_{\Lambda}\to\mathbb{T}^{d}$ we can write $\alpha$ as a ${\rm Diff}^{1}(\mathbb{T})-$valued cocycle $\beta$ over $\rho$. Note that the action $\rho$ is given by toral automorphisms. By Theorem \ref{Thm:SolvingDiffTValuedCocycle} and the discussion below, the cocycle is cohomologous by $h:\mathbb{T}^{d}\to{\rm Diff}(\mathbb{T})$ to a cocycle taking values in ${\rm Iso}(\mathbb{T})$. Let $\mu_{y}^{c} = h(y)_{*}^{-1}{\rm vol}_{\mathbb{T}}$, $y\in\mathbb{T}^{d}$, then
\begin{align}
\mu = \int_{\mathbb{T}^{d}}\mu_{y}^{c}\intd{\rm vol}_{\mathbb{T}^{d}}(y)
\end{align}
is $\alpha-$invariant and has conditionals equivalent to Lebesgue. Let $f = \alpha(\gamma_{0})$ be the partially hyperbolic element. The measure $\mu$ is $f-$invariant, has zero center Lyapunov exponent (since the conditionals along the center are Lebesgue), and projects to ${\rm vol}_{\mathbb{T}^{d}}$. It follows by the invariance principle \cite{AvilaViana2010}*{Theorem D} that the disintegration $\mu_{x}^{c}$ along $\Phi$ are $su-$holonomy invariant, and, in particular, vary continuously in $x$.
\end{proof}
\begin{lemma}\label{L:CenterTranslationAction}
There is a continuous $\mathbb{T}-$action $\eta_{c}:\mathbb{T}\times X_{\Lambda}\to X_{\Lambda}$ that commutes with $\alpha$ and act transitively and freely on each center leaf $W^{c}(x)$.
\end{lemma}
\begin{proof}
Let $\mu_{y}^{c}$, $y\in\mathbb{T}^{d}$, be as in the previous lemma. There is a unique orientation preserving, $\mu_{y}^{c}-$preserving homeomorphism $h_{s}:W^{c}(x)\to W^{c}(x)$, $\Phi(x) = y$, such that the rotation number of $h_{s}$ is $s$. We define $\eta_{c}(s)x = h_{s}(x)$. It is immediate that $\eta_{c}$ defines a $\mathbb{T}-$action, that it is continuous follows since the disintegration $\mu_{y}^{c}$ vary continuously.
\end{proof}
\begin{lemma}\label{L:InvariantCenterMetric}
There is a metric $\intd_{c}$ on $W^{c}$ that is $\eta_{c}$ and $\alpha-$inariant.
\end{lemma}
\begin{proof}
The conditionals of $\mu$, $\mu_{x}^{c}$, are equivalent to Lebesgue and therefore defines a metric $\omega^{c}$ on the bundle $E^{c}$.Let $\intd_{c}$ be the metric along $W^{c}$ defined by $\omega^{c}$. That $\eta_{c}$ and $\alpha$ both preserve $\intd_{c}$ follows since both actions preserve $\mu$.
\end{proof}

\subsection{Topological rigidity for a Cartan subaction}\label{SubSec:ConjugacyForAbelianSubgroup}

In the remainder of this section, let $\eta_{c}:\mathbb{T}\times X_{\Lambda}\to X_{\Lambda}$ be the action from Lemma \ref{L:CenterTranslationAction}. We will denote by $\intd_{c}$ the metric from Lemma \ref{L:InvariantCenterMetric}, note that this metric is induced by a Riemannian metric so is comparable to the standard metric for small distances. Moreover, $\intd_{c}$ is $\alpha-$invariant by Lemma \ref{L:InvariantCenterMetric}.
Let $\mathbb{Z}^{k}\cong\Sigma\subset\Gamma$ be an abelian subgroup such that ${\rm Im}(\rho|_{\Sigma})\subset{\rm GL}(d,\mathbb{Z})$ is higher rank and contains a hyperbolic matrix. Such a subgroup always exists by Lemma \ref{L:ExistenceOfHigherRankAbelianGroup}.

We will denote by $\Delta_{\Sigma}$ the coarse exponents of $\rho|_{\Sigma}$. Given $[\chi]\in\Delta_{\Sigma}$ let $E_{\Sigma}^{[\chi]}\subset\mathbb{R}^{d}$ be the corresponding coarse subspace.

The following lemma is immediate since $\Phi:X_{\Lambda}\to\mathbb{T}^{d}$ is bi-Hölder along $W^{s}$ and $W^{u}$ \cite{Sandfeldt2024}*{Lemma 3.3} (where $W^{s}$ and $W^{u}$ are the stable and unstable foliations for the partially hyperbolic element $f$).
\begin{lemma}\label{L:LiftOfClosePointsFromBase}
There is $\varepsilon_{0} > 0$, $\beta\in(0,1)$, and $C\geq1$ such that for any $x\in X_{\Lambda}$ and $y\in B_{\varepsilon_{0}}(x)$ there is $p\in W^{c}(y)$ satisfying $\intd(x,p)\leq C\intd(\Phi(x),\Phi(y))^{\beta}$.
\end{lemma}
\begin{lemma}\label{L:ExistenceOfLiftedAction}
For any $[\chi]\in\Delta_{\Sigma}$, every $v\in E_{\Sigma}^{[\chi]}$, and every $x\in X_{\Lambda}$ there is a unique $\eta_{[\chi]}(v)x\in X_{\Lambda}$ such that $\Phi(\eta_{[\chi]}(v)x) = \Phi(x) + v$ and 
\begin{align}
\lim_{n\to\infty}\intd(\alpha(a)^{n}x,\alpha(a)^{n}\eta_{[\chi]}(v)x) = 0,
\end{align}
for any $a\in\Sigma$ such that $[\chi](a) < 0$. Moreover, the point $\eta_{[\chi]}(v)x$ depends continuously on $v\in E_{\Sigma}^{[\chi]}$, where the continuity is uniform in $x$.
\end{lemma}
\begin{proof}
Fix $a\in\Sigma$ such that $[\chi](a) < 0$ and let $g = \alpha(a)$. Fix $x\in X_{\Lambda}$ and $v\in E_{\Sigma}^{[\chi]}$. After possibly exchanging $x$ by $g^{n}x$ and $v$ by $\rho(a)^{n}v$ we may assume that $\norm{v} < \varepsilon_{0}$. We fix $\lambda\in(0,1)$ such that $\norm{\rho(a)|_{E_{\Sigma}^{[\chi]}}}\leq\lambda$.

Let $\Phi^{-1}(\Phi(x) + v) = W^{c}(y)$. Since we assume that $\norm{v} < \varepsilon_{0}$ we find $p_{0}\in W^{c}(y)$ such that $\intd(x,p_{0}) < \varepsilon_{0}$. Write $x_{n} = g^{n}x$ and let $p_{n}(v) = p_{n}\in W^{c}(g^{n}p_{0})$ be chosen such that $\intd(x_{n},p_{n})\leq C\intd(\Phi(x),\Phi(p_{0}))^{\beta}\tau^{n}$ with $\tau = \lambda^{\beta}\in(0,1)$. Such a choice is possible by Lemma \ref{L:LiftOfClosePointsFromBase} since 
\begin{align}
\intd(\Phi(x_{n}),\Phi(g^{n}p_{0})) = \intd(\rho(a)^{n}\Phi(x),\rho(a)^{n}\Phi(x) + \rho(a)^{n}v)\leq\lambda^{n}\cdot\norm{v}.
\end{align}
It is immediate
\begin{align*}
\intd_{c}(gp_{n-1},p_{n})\leq & \intd(gp_{n-1},gx_{n-1}) + \intd(p_{n},x_{n})\leq \\ &
\norm{Df}\cdot\intd(p_{n-1},x_{n-1}) + \intd(p_{n},x_{n})\leq \\ &
C\left[\frac{\norm{Df}}{\tau} + 1\right]\intd(\Phi(x),\Phi(p_{0}))^{\beta}\tau^{n} =: \\ &
K\intd(\Phi(x),\Phi(p_{0}))^{\beta}\tau^{n}.
\end{align*}
Given $m\leq n$ we obtain by induction
\begin{align}\label{Eq:InductiveEstimate}
\intd_{c}(g^{m}p_{n-m},p_{n})&\leq\intd_{c}(g^{m}p_{n-m},gp_{n-1}) + \intd_{c}(p_{n},gp_{n-1})\leq \\ &
\intd_{c}(g^{m-1}p_{n-m},p_{n-1}) + K\intd(\Phi(x),\Phi(p_{0}))^{\beta}\tau^{n}\leq...\leq \\ &
K\intd(\Phi(x),\Phi(p_{0}))^{\beta}\sum_{j = n-m+1}^{n}\tau^{j} = K'\intd(\Phi(x),\Phi(p_{0}))^{\beta}\tau^{n-m},
\end{align}
where we used that $g$ is isometric with respect to $\intd_{c}$. If $q_{n} = g^{-n}p_{n}\in W^{c}(y)$ then
\begin{align*}
\intd_{c}(q_{n},q_{n+m}) = & \intd_{c}(g^{-n}p_{n},g^{-n-m}p_{n+m}) = \intd_{c}(g^{m}p_{n},p_{n+m})\leq \\ & K'\intd(\Phi(x),\Phi(p_{0}))^{\beta}\tau^{n}
\end{align*}
so the sequence $q_{n}$ is Cauchy. If we denote the limit $q_{\infty} = \lim_{n\to\infty}q_{n}$ then
\begin{align*}
\intd(g^{n}q_{\infty},g^{n}x)\leq & \intd(g^{n}q_{\infty},g^{n}q_{n}) + \intd(g^{n}q_{n},g^{n}x)\leq \\ &
\intd_{c}(g^{n}q_{\infty},g^{n}q_{n}) + \intd(p_{n},g^{n}x)\leq \\ &
K'\intd(\Phi(x),\Phi(p_{0}))^{\beta}\tau^{n} + C\intd(\Phi(x),\Phi(p_{0}))^{\beta}\tau^{n} = \\ &
K''\intd(\Phi(x),\Phi(p_{0}))^{\beta}\tau^{n}
\end{align*}
so $\intd(g^{n}q_{\infty},g^{n}x)\to0$. Finally, we will prove that $\eta_{[\chi]}(v)x$ is independent of which $a$, $[\chi](a) < 0$, was chosen by proving that $v\mapsto\eta_{[\chi]}(v)x$ is continuous in $v$. Indeed, for any other $a'\in\Sigma$ we have $\alpha(a')\eta_{[\chi]}(v)x = \eta_{[\chi]}(\rho(a')v)\alpha(a')x$ since $\alpha(a')$ commute with $g$ and $\Phi(\alpha(a')\eta_{[\chi]}(v)x) = \rho(a')\Phi(x) + \rho(a')v = \Phi(\alpha(a')x) + \rho(a')v$. Since $\rho(a')^{n}v\to0$ it follows that $\intd(\alpha(a')^{n}\eta_{[\chi]}(v)x,\alpha(a')^{n}x)\to0$, provided that $\eta_{[\chi]}(v)x$ is continuous in $v$, which shows that the construction of $\eta_{[\chi]}(v)x$ is independent of which $a\in\Sigma$ was used to construct it.

From Equation \ref{Eq:InductiveEstimate} we have
\begin{align}
\intd_{c}(p_{0}(v),p_{n}(v))\leq K'\intd(\Phi(x),\Phi(p_{0}(v)))^{\beta}
\end{align}
or after letting $n\to\infty$
\begin{align}
\intd_{c}(p_{0}(v),\eta_{[\chi]}(v)x)\leq K'\intd(\Phi(x),\Phi(p_{0}(v)))^{\beta} = K'\norm{v}^{\beta}.
\end{align}
On the other hand, we chose $p_{0}(v)$ such that $\intd(x,p_{0}(v))\leq C\intd(\Phi(x),\Phi(x)+v)^{\beta} = C\norm{v}^{\beta}$ so by the triangle inequality we have $\intd(x,\eta_{[\chi]}(v)x)\leq c\norm{v}^{\beta}$ for some constant $c > 0$. The triangle inequality immediately implies that $\eta_{[\chi]}(v+v')x = \eta_{[\chi]}(v)\eta_{[\chi]}(v')x$ so for $v,v'\in E_{\Sigma}^{[\chi]}$
\begin{align}
\intd(\eta_{[\chi]}(v)x,\eta_{[\chi]}(v')x) = \intd(\eta_{[\chi]}(v-v')\eta_{[\chi]}(v')x,\eta_{[\chi]}(v')x)\leq c\norm{v-v'}^{\beta}
\end{align}
proving that $v\mapsto\eta_{[\chi]}(v)x$ is continuous in $v$ (uniformly in $x$).
\end{proof}
\begin{lemma}\label{L:ContinuityLiftedAction}
The function $\eta_{[\chi]}:E_{\Sigma}^{[\chi]}\times X_{\Lambda}\to X_{\Lambda}$, $\eta_{[\chi]}(v)x$, is a continuous $E_{\Sigma}^{[\chi]}-$action covering the translation action along $E_{\Sigma}^{[\chi]}$ on $\mathbb{T}^{d}$.
\end{lemma}
\begin{proof}
We already proved that $\eta_{[\chi]}$ defines an $E_{\Sigma}^{[\chi]}-$action in the proof of Lemma \ref{L:ExistenceOfLiftedAction}. That $\eta_{[\chi]}$ covers the translation action on $\mathbb{T}^{d}$ is immediate from the definition. Since $\eta_{[\chi]}(v)x$ is (uniformly in $x$) continuous in $v$, by Lemma \ref{L:ExistenceOfLiftedAction}, it remains to prove that $\eta_{[\chi]}(v)x$ is continuous in $x$.

As in the proof of Lemma \ref{L:ExistenceOfLiftedAction}, we fix $a\in\Sigma$ such that $[\chi](a) < 0$ and let $g = \alpha(a)$. Assume that $\eta_{[\chi]}(v)$ is not continuous at $x\in X_{\Lambda}$. We find a sequence $x_{n}\to x$ such that $\intd(\eta_{[\chi]}(v)x_{n},\eta_{[\chi]}(v)x)\geq c > 0$. From the definition of $\eta_{[\chi]}(v)$ we have 
\begin{align}
\intd(\Phi(\eta_{[\chi]}(v)x_{n}),\Phi(\eta_{[\chi]}(v)x))\leq\intd(\Phi(x_{n}),\Phi(x))\to0,
\end{align}
so by Lemma \ref{L:LiftOfClosePointsFromBase} we find $y_{n}\in W^{c}(\eta_{[\chi]}(v)x)$ such that $\lim_{n\to\infty}\intd(\eta_{[\chi]}(v)x_{n},y_{n}) = 0$. For $n$ sufficiently large we obtain
\begin{align}
\intd(y_{n},\eta_{[\chi]}(v)x)\geq c - \intd(\eta_{[\chi]}(v)x_{n},y_{n})\geq c/2 > 0.
\end{align}
As in the proof of Lemma \ref{L:ExistenceOfLiftedAction} we have
\begin{align*}
\intd(g^{m}x,g^{m}\eta_{[\chi]}(v)x),\intd(g^{m}x_{n},g^{m}\eta_{[\chi]}(v)x_{n})\leq K'\norm{v}^{\beta}\tau^{m},\quad\tau\in(0,1).
\end{align*}
For $m\geq0$ we have
\begin{align*}
& \intd(g^{m}\eta_{[\chi]}(v)x_{n},g^{m}\eta_{[\chi]}(v)x)\leq \\ &
\intd(g^{m}\eta_{[\chi]}(v)x,g^{m}x) + \intd(g^{m}x,g^{m}x_{n}) + \intd(g^{m}\eta_{[\chi]}(v)x_{n},g^{m}x_{n})\leq \\ &
2K'\norm{v}^{\beta}\tau^{m} + \intd(g^{m}x,g^{m}x_{n})
\end{align*}
and
\begin{align*}
\intd(g^{m}\eta_{[\chi]}(v)x_{n},g^{m}\eta_{[\chi]}(v)x)\geq & \intd_{c}(g^{m}y_{n},g^{m}\eta_{[\chi]}(v)x) - \intd(g^{m}y_{n},g^{m}\eta_{[\chi]}(v)x_{n})\geq \\ &
\frac{c}{2} - \intd(g^{m}y_{n},g^{m}\eta_{[\chi]}(v)x_{n})
\end{align*}
so combined
\begin{align}\label{Eq:ContradictionEquation}
\frac{c}{2}\leq 2K'\norm{v}^{\beta}\tau^{m} + \intd(g^{m}x,g^{m}x_{n}) + \intd(g^{m}y_{n},g^{m}\eta_{[\chi]}(v)x_{n}).
\end{align}
With $m$ large and fixed such that $2K'\norm{v}^{\beta}\tau^{m} < c/4$, Equation \ref{Eq:ContradictionEquation} gives a contradiction as $n\to\infty$ since $\intd(x,x_{n}),\intd(y_{n},\eta_{[\chi]}(v)x_{n})\to0$ and $g^{m}$ is continuous.
\end{proof}
\begin{lemma}\label{L:PropertiesOfLiftedTranslarionAction}
The following properties hold for each $[\chi]\in\Delta_{\Sigma}$
\begin{enumerate}[label = (\roman*)]
    \item for $a\in\Sigma$ we have $\alpha(a)\eta_{[\chi]}(v) = \eta_{[\chi]}(\rho(a)v)\alpha(a)$,
    \item the action $\eta_{[\chi]}$ commute with the center action $\eta_{c}$,
    \item the action $\eta_{[\chi]}$ preserve the $\Gamma-$invariant measure $\mu$.
\end{enumerate}
In particular, the maps $\eta_{[\chi]}(v)$ preserve the $W^{c}-$foliation.
\end{lemma}
\begin{proof}
Property $(iii)$ follows from $(ii)$. Properties $(i)$ and $(ii)$ are immediate from the definition since $\eta_{c}$ commute with $\alpha$ and since $\Phi(\alpha(a)x) = \rho(a)\Phi(x)$.
\end{proof}
\begin{lemma}
Given $v_{1}\in E_{\Sigma}^{[\chi]}$ and $v_{2}\in E_{\Sigma}^{[\lambda]}$ with $[\lambda]\neq\pm[\chi]$ we have $\eta_{[\chi]}(v_{1})\eta_{[\lambda]}(v_{2}) = \eta_{[\lambda]}(v_{2})\eta_{[\chi]}(v_{1})$.
\end{lemma}
\begin{proof}
Choose some $a\in\Sigma$ such that $\rho(a)$ contracts $E_{\Sigma}^{[\chi]}$ and $E_{\Sigma}^{[\lambda]}$ simultaneously. Let $g = \alpha(a)$. It holds that 
\begin{align}
\intd(g^{n}\eta_{[\chi]}(\pm v_{1})x,g^{n}x),\intd(g^{n}\eta_{[\lambda]}(\pm v_{2})x,g^{n}x)\to0,
\end{align}
for all $x\in X_{\Lambda}$. For any $x\in X_{\Lambda}$ we have $y = \eta_{[\lambda]}(-v_{2})\eta_{\chi}(-v_{1})\eta_{\lambda}(v_{2})\eta_{[\chi]}(v_{1})x\in W^{c}(x)$ since the projection onto $\mathbb{T}^{d}$ is $\Phi(x) + v_{1} + v_{2} - v_{1} - v_{2} = \Phi(x)$. Using that $g$ is isometric along $W^{c}$
\begin{align*}
\intd(x,y) = & \intd(g^{n}x,g^{n}y)\leq \\ &
\intd(g^{n}\eta_{[\lambda]}(-v_{2})\eta_{[\chi]}(-v_{1})\eta_{[\lambda]}(v_{2})\eta_{[\chi]}(v_{1})x,g^{n}\eta_{[\chi]}(-v_{1})\eta_{[\lambda]}(v_{2})\eta_{[\chi]}(v_{1})x) + \\
& \intd(g^{n}\eta_{[\chi]}(-v_{1})\eta_{[\lambda]}(v_{2})\eta_{[\chi]}(v_{1})x,g^{n}\eta_{[\lambda]}(v_{2})\eta_{[\chi]}(v_{1})x) + \\
& \intd(g^{n}\eta_{[\lambda]}(v_{2})\eta_{[\chi]}(v_{1})x,g^{n}\eta_{[\chi]}(v_{1})x) + \\
& \intd(g^{n}\eta_{[\chi]}(v_{1})x,g^{n}x)
\end{align*}
and letting $n\to\infty$ we obtain $\intd(x,y) = 0$.
\end{proof}
\begin{lemma}\label{L:SymplecticPairRelations}
Given $v\in E_{\Sigma}^{[\chi]}$ and $w\in E_{\Sigma}^{-[\chi]}$ there is $s(v,w)\in\mathbb{T}$ such that
\begin{align}
\eta_{-[\chi]}(-w)\eta_{[\chi]}(-v)\eta_{-[\chi]}(w)\eta_{[\chi]}(v) = \eta_{c}(s(v,w)).
\end{align}
\end{lemma}
\begin{proof}
We can write 
\begin{align}
\eta_{-[\chi]}(-w)\eta_{[\chi]}(-v)\eta_{-[\chi]}(w)\eta_{[\chi]}(v)x = \eta_{c}(s_{\Phi(x)}(v,w))x
\end{align}
since $\eta_{-[\chi]}(-w)\eta_{[\chi]}(-v)\eta_{-[\chi]}(w)\eta_{[\chi]}(v)$ fix center leaves and commute with $\eta_{c}$. The map $\mathbb{T}^{d}\ni y\mapsto s_{y}(v,w)\in\mathbb{T}$ is continuous and, as in \cite{Sandfeldt2024}*{Lemma 6.9}, $s_{y}(v,w)$ is $E_{\Sigma}^{[\lambda]}-$translation invariant for every $[\lambda]\neq\pm[\chi]$. By Definition \ref{Def:HigherRankAbelianAction} (or \cite{Sandfeldt2024}*{Lemma 2.2}) the space 
\begin{align}
V = \bigoplus_{[\lambda]\neq\pm[\chi]}E_{\Sigma}^{[\lambda]}
\end{align}
has a minimal translation action on $\mathbb{T}^{d}$, so $s_{y}(v,w)$ is independent of $y$.
\end{proof}
The proof of the following Theorem is essentially contained in \cite{Sandfeldt2024}*{Section 7}.
\begin{theorem}
There is a homeomorphism $H:X_{\Lambda}\to X_{\Lambda}$ such that $H\circ\alpha|_{\Sigma}\circ H^{-1} = \alpha_{0}|_{\Sigma}$ is an affine action. Moreover, $H$ satisfy $H(\eta_{c}(s)x) = H(x)e^{sZ}$ where $Z$ is a generator of $E_{0}^{c}$.
\end{theorem}
\begin{proof}
For each $[\chi]\in\Delta_{\Sigma}$ there is a unique lift of $\eta_{[\chi]}$ to an action $\Tilde{\eta}_{[\chi]}:E_{\Sigma}^{[\chi]}\times N\to N$. We also have a unique lift of $\eta_{c}$ to an action $\Tilde{\eta}_{c}:\mathbb{R}\times N\to N$. Finally, we also lift $\Phi:X_{\Lambda}\to\mathbb{T}^{d}$ to a fibration $\Tilde{\Phi}:N\to\mathbb{R}^{d}$. Denote by $\Tilde{N}\subset{\rm Homeo}(N)$ the group generated by $(\Tilde{\eta}_{[\chi]})_{[\chi]\in\Delta_{\Sigma}}$ and $\Tilde{\eta}_{c}$. Note that $\Tilde{\Phi}(\Tilde{\eta}_{[\chi]}(v)x) = \Tilde{\Phi}(x) + v$ so, since the translation action on $\mathbb{R}^{d}$ is transitive (as a group action), it holds that for every $x,y\in N$ there is some $n\in\Tilde{N}$ such that $n(x)\in\Tilde{W}^{c}(y)$ (where $\Tilde{W}^{c}$ is the lifted center foliation). Moreover, $\eta_{c}$ acts transitively on each $W^{c}-$leaf so $\Tilde{\eta}_{c}$ act transitively on each leaf of $\Tilde{W}^{c}$. Combined this implies that $\Tilde{N}$ act transitively on $N$.

Fix an order of the coarse exponents $\Delta_{\Sigma} = \{[\chi_{1}],...,[\chi_{r}]\}$. Denote by $P_{[\chi]}:\mathbb{R}^{d}\to E_{\Sigma}^{[\chi]}$ the projection map. By Lemma \ref{L:SymplecticPairRelations} and property $(ii)$ in Lemma \ref{L:PropertiesOfLiftedTranslarionAction} it is clear that the map
\begin{align}
P:\mathbb{R}^{d}\times\mathbb{R}\to\Tilde{N},\quad P(v,s) := \eta_{[\chi_{1}]}(P_{[\chi_{1}]}v)...\eta_{[\chi_{r}]}(P_{[\chi_{r}]}v)\eta_{c}(s)
\end{align}
is surjective. Suppose that there is $x\in N$ such that $P(v,s)x = x$ for some $v\in\mathbb{R}^{d}$ and $s\in\mathbb{R}$. That is, $\Tilde{\eta}_{[\chi_{1}]}(P_{[\chi_{1}]}v)...\Tilde{\eta}_{[\chi_{r}]}(P_{[\chi_{r}]}v)\Tilde{\eta}_{c}(s)x = x$, so after applying $\Tilde{\Phi}$ on both sides and using that $\eta_{[\chi_{j}]}$ cover a translation action for each $j$ we obtain
\begin{align}
\Tilde{\Phi}(x) + P_{[\chi_{1}]}v+...+P_{[\chi_{r}]}v = \Tilde{\Phi}(x) + v = \Tilde{\Phi}(x)
\end{align}
which implies that $v = 0$. Therefore, if $P(v,s)x = x$ then $v = 0$ so $\Tilde{\eta}_{c}(s)x = x$. Since $\eta_{c}$ acts freely the action $\Tilde{\eta}_{c}$ is free which implies that $s = 0$. That is, if $P(v,s)$ has a fixed point then $v = 0$, $s = 0$, and $P(v,s) = {\rm id}$ so $\Tilde{N}$ act freely on $N$. A similar argument also shows that $P:\mathbb{R}^{d}\times\mathbb{R}\to\Tilde{N}$ is injective, so its a bijection. Moreover, if $P(v,s)$ is close to $P(v',s')$ then $\Tilde{\Phi}(x) + v = \Tilde{\Phi}(P(v,s)x)\approx\Tilde{\Phi}(P(v',s')x) = \Tilde{\Phi}(x) + v'$ so $v$ is close to $v'$. Note that we can write $P(v',0)^{-1}P(v,0) = P(v-v',t(v,v'))$ for some $t(v,v')$ (that depends continuously on $v$ and $v'$). Now, we can write
\begin{align*}
n_{v,v'} = & P(v',s')^{-1}P(v,s) = P(v',0)^{-1}P(0,-s')P(v,0)P(0,s) = \\ &
P(v-v',0)P(0,s-s' + t(v,v'))
\end{align*}
but since $n_{v,v'}$ and $P(v-v',0)$ are both close to identity this implies that $s-s' + t(v,v')$ is close to $0$. Since $t(v,v) = 0$ and $t(v,v')$ vary continuously in $v$ and $v'$ this implies that $t(v,v')$ is close to $0$. So, $s - s'$ is close to zero or equivalently, $s$ is close to $s'$. That is, if $P(v,s)$ is close to $P(v',s')$ then $(v,s)$ is close to $(v',s')$. This implies, since $P$ is bijective, that $P$ is a homeomorphism. The topological group $\Tilde{N}$ is homeomorphic to the Euclidean space $\mathbb{R}^{d}\times\mathbb{R}$ and is $2-$step nilpotent as an abstract group, so $\Tilde{N}$ is a $2-$step nilpotent Lie group \cites{Gleason1952,MontgomeryZippin1952}.

That is, $\Tilde{N}$ is a $2-$step nilpotent Lie group that acts transitively and freely on $N$. We fix $x_{0}\in N$. By invariance of domain the map $Q:\Tilde{N}\to N$, $Q(n) := n(x_{0})$, is a homeomorphism. Property $(i)$ from Lemma \ref{L:PropertiesOfLiftedTranslarionAction} and Lemma \ref{L:CenterTranslationAction} imply that $\alpha(a)\Tilde{N} = \Tilde{N}\alpha(a)$ so $n\mapsto\alpha(a)n\alpha(a)^{-1}$ is a (continuous) automorphism, we denote this automorphism by $\Tilde{\alpha}_{0}(a):\Tilde{N}\to\Tilde{N}$. We have
\begin{align*}
\alpha(a)Q(n) = & \alpha(a)n(x_{0}) = (\Tilde{\alpha}_{0}(a)n)(\alpha(a)x_{0}) = (\Tilde{\alpha}_{0}(a)n)\circ n_{a}(x_{0}) = \\ &
Q(\Tilde{\alpha}_{0}(a)(n)\cdot n_{a}) =: Q(\alpha_{0}(a)(n))
\end{align*}
where $n_{a}\in\Tilde{N}$ is chosen such that $n_{a}(x_{0}) = \alpha(a)x_{0}$ (this defines $n_{a}$ uniquely since $\Tilde{N}$ act transitively and freely). It follows that $Q$ conjugates $\Tilde{\alpha}|_{\Sigma}$ to an affine action on $\Tilde{N}$. Proving that $Q$ defines a homeomorphism $H:X_{\Lambda}\to X_{\Lambda}$ follows as in \cite{Sandfeldt2024}*{Theorem 7.1}.

The final part of the theorem is immediate from our construction of $Q$. Indeed, we have $\eta_{c}(s)Q(n) = \eta_{c}(s)n(x_{0}) = (n\eta_{c}(s))x_{0} = Q(n\eta_{c}(s))$, so $Q^{-1}$ conjugates $\eta_{c}(s)$ to some translation action. That $\eta_{c}(s)$ coincides with $E_{0}^{c}$ is immediate since $\alpha$ commute with $\eta_{c}$ (Lemma \ref{L:PropertiesOfLiftedTranslarionAction} property $(ii)$).
\end{proof}

\subsection{Extending the conjugacy to the full lattice}\label{SubSec:ExtensionOfTopologicalConjugacy}

Next, we will extend the conjugacy $H:X_{\Lambda}\to X_{\Lambda}$ from $\Sigma$ to a conjugacy for all of $\Gamma$. Let $\alpha_{0}(\gamma) := H\alpha(\gamma)H^{-1}$, then $\alpha_{0}$ is a $\Gamma$ action such that $\alpha_{0}|_{\Sigma}$ is an action by affine maps. We denote by $\Hat{\rho}:\Gamma\to{\rm Aut}(X_{\Lambda})$ the induced map on fundamental group. It is immediate that $\pi_{*}\Hat{\rho}(\gamma) = \rho(\gamma)\pi_{*}$. Let $\Tilde{\alpha}_{0}:\Gamma\times N\to N$ be a lift of $\alpha_{0}$ to $N$. Define a map $C:\Gamma\times N\to\mathfrak{n}$ by $\Tilde{\alpha}_{0}(\gamma)x = \Hat{\rho}(\gamma)x\cdot e^{C(\gamma,x)}$. From the construction of $H$, and the fact that $\eta_{[\chi]}$ cover the corresponding translation action on $\mathbb{T}^{d}$ (Lemma \ref{L:ContinuityLiftedAction}), it is clear that $\pi\circ H = \Phi$. It follows that $\pi:N\to N/[N,N]$ satisfies $\pi\Tilde{\alpha}_{0}(\gamma)x = \rho(\gamma)\pi(x)$ so the map $C:\Gamma\times N\to\mathfrak{n}$ takes values in $\ker D\pi = E_{0}^{c}$. We have
\begin{align*}
\Hat{\rho}(\gamma\gamma')x\cdot e^{C(\gamma\gamma',x)} = & \Tilde{\alpha}(\gamma\gamma')x = \Tilde{\alpha}(\gamma)\Tilde{\alpha}(\gamma')x = \Hat{\rho}(\gamma)(\Tilde{\alpha}(\gamma')x)\cdot e^{C(\gamma,\Tilde{\alpha}_{0}(\gamma')x)} = \\ &
\Hat{\rho}(\gamma)\Hat{\rho}(\gamma')x\cdot e^{\Hat{\rho}(\gamma)C(\gamma',x) + C(\gamma,\Tilde{\alpha}_{0}(\gamma')x)} = \\ &
\Hat{\rho}(\gamma\gamma')x\cdot e^{C(\gamma',x) + C(\gamma,\Tilde{\alpha}_{0}(\gamma')x)}
\end{align*}
where we have used that $\Hat{\rho}(\gamma)$ acts trivially on $E_{0}^{c}$. It follows that 
\begin{align*}
C(\gamma\gamma',x) = C(\gamma,\Tilde{\alpha}_{0}(\gamma')x) + C(\gamma',x)
\end{align*}
so $C$ is a $E_{0}^{c}-$valued cocycle over the action $\Tilde{\alpha}_{0}$. Moreover, for $\lambda\in\Lambda$ we have $\Tilde{\alpha}_{0}(\gamma)\lambda x = \Hat{\rho}(\gamma)\lambda\cdot\Tilde{\alpha}_{0}(\gamma)x$ so
\begin{align}
\Hat{\rho}(\gamma)(\lambda x)\cdot e^{C(\gamma,\lambda x)} = \Hat{\rho}(\gamma)\lambda\cdot\Hat{\rho}(\gamma)x\cdot e^{C(\gamma,x)}
\end{align}
or $C(\gamma,\lambda x) = C(\gamma,x)$. It follows that $C$ descend to a map $C:\Gamma\times X_{\Lambda}\to E_{0}^{c}$. That is, $C$ is a $E_{0}^{c}-$valued cocycle over $\alpha_{0}$.
\begin{lemma}\label{L:CharacterizationOfDefectCocycle}
The cocycle $C$ factor through a cocycle $Z:\Gamma\times\mathbb{T}^{d}\to E_{0}^{c}$.
\end{lemma}
\begin{proof}
Since $H(\eta_{c}(s)x) = H(x)e^{sZ}$ we obtain
\begin{align*}
\alpha_{0}(\gamma)(xe^{sZ}) = \Hat{\rho}(\gamma)x\cdot e^{sZ}\cdot e^{C(\gamma,xe^{sZ})} = \alpha_{0}(\gamma)x\cdot e^{sZ} = \Hat{\rho}(\gamma)x\cdot e^{C(\gamma,x)}\cdot e^{sZ}
\end{align*}
or $C(\gamma,xe^{sZ}) = C(\gamma,x)$. So $C$ descends to a map $Z:\Gamma\times\mathbb{T}^{d}\to E_{0}^{c}$.
\end{proof}
\begin{theorem}\label{Thm:TopologicalRigidity}
The action $\alpha_{0}$ is affine, so $\alpha$ is topologically conjugated to an action by affine maps.
\end{theorem}
\begin{proof}
The maps $Z(\gamma,x)$ from Lemma \ref{L:CharacterizationOfDefectCocycle} is a $E_{0}^{c}-$valued cocycle over the action $\rho$. By Theorem \ref{Thm:MeasurableTrivializationOfCocycle} the cocycle $Z$ is measurably trivial. That is, there is a measurable function $\varphi:\mathbb{T}^{d}\to E_{0}^{c}$ such that $Z(\gamma,x) = \varphi(\rho(\gamma)x) - \varphi(x)$ volume almost everywhere. Fix $a\in\Sigma$ such that $\rho(a) = L:\mathbb{T}^{d}\to\mathbb{T}^{d}$ is ergodic. Since $H$ conjugates $\alpha(a)$ to an affine map $Z(a,x)$ is constant. We claim that this constant is $0$. To see this, let
\begin{align}
\psi:\Gamma\to E_{0}^{c},\quad\psi(\gamma) := \int_{\mathbb{T}^{d}}Z(\gamma,x)\intd x.
\end{align}
The fact that $Z$ is a cocycle implies that $\psi$ is a homomorphism, but $E_{0}^{c}$ is a vector space so $\psi = 0$. It follows that $Z(a,x)$ is constant and integrates to $0$, so $Z(a,x) = 0$. That is, $\varphi(Lx) - \varphi(x) = 0$ volume almost everywhere. Since $\varphi$ is $L-$invariant and $L$ is ergodic we conclude that $\varphi$ is constant so $Z(\gamma,x) = 0$ is constant equal to $0$. This finish the proof of the theorem since $Z = 0$ implies that $C = 0$ so $\alpha_{0}$ coincide with $\Hat{\rho}$.
\end{proof}

\section{Smooth rigidity: Non-accessible actions}

We now start the proof of Theorem \ref{MainThm:AbelianCase}. For the remainder of this section let $X_{\Lambda} = \mathbb{T}^{d}\times\mathbb{T}$. By Theorem \ref{Thm:TopologicalRigidity} the action $\alpha$ is topologically conjugated to an affine action $\alpha_{0}:\Gamma\times X_{\Lambda}\to X_{\Lambda}$. For any affine partially hyperbolic diffeomorphism on a torus the stable and unstable distributions are jointly integrable. It follows that there is a (topological) foliation $\mathcal{V}$ tangent to $V = E^{s}\oplus E^{u}$ (where we recall that $E^{\sigma}$ is the invariant distribution for the partially hyperbolic element $\alpha(\gamma_{0}) = f$). The following lemma is well-known.
\begin{lemma}\label{L:SmoothHorizontalFoliation}
The foliation $\mathcal{V}$ is a $C^{\infty}-$foliation.
\end{lemma}
\begin{proof}
By \cite{GogolevKalininSadovskaya2022}*{Lemma 4.1}, and the discussion after the lemma, $\mathcal{V}$ is a $C^{\infty}-$foliation (see also \cite{KalininSadovskaya2006}*{Lemma 4.1}).
\end{proof}
\begin{lemma}\label{L:GlobalProductStructure}
For any $x,y\in X_{\Lambda}$ we have $\#\mathcal{V}(x)\cap W^{c}(y) = 1$. In particular, $\mathcal{V}$ have compact leaves
\end{lemma}
\begin{proof}
By Theorem \ref{Thm:TopologicalRigidity} the action $\alpha$ is topologically conjugated to an affine action $\alpha_{0}$ with linear part taking values in ${\rm GL}(d,\mathbb{Z})\times 1$. Since the lemma holds for $\alpha_{0}$, the lemma also holds for $\alpha$.
\end{proof}
\begin{lemma}
The action $\alpha$ preserve $\mathcal{V}$.
\end{lemma}
\begin{proof}
This is immediate from Theorem \ref{Thm:TopologicalRigidity}.
\end{proof}
\begin{lemma}\label{L:SmoothSubmersion}
The projection map $\Phi:X_{\Lambda}\to\mathbb{T}^{d}$ is a $C^{\infty}-$submersion.
\end{lemma}
\begin{proof}
By Lemma \ref{L:GlobalProductStructure} we have $\#\mathcal{V}(x)\cap W^{c}(y) = 1$ for all $x,y\in X_{\Lambda}$. Since $\Gamma$ is a higher rank lattice the action of $\Gamma$ on a circle is virtually trivial \cites{DeroinHurtado,Navas}. So, after possibly dropping to a finite index subgroup, we may assume that $\alpha$ fixes all leaves of $\mathcal{V}$.

For every $x\in X_{\Lambda}$ the action $\alpha_{x}:\Gamma\times\mathcal{V}(x)\to\mathcal{V}(x)$ is a higher rank Anosov action (since the partially hyperbolic element $f = \alpha(\gamma_{0})$ restrict to an Anosov diffeomorphism on $\mathcal{V}(x)$). By \cite{BrownRodriguez-HertzWang2017} the map $\Phi|_{\mathcal{V}(x)} = \Phi_{x}:\mathcal{V}(x)\to\mathbb{T}^{d}$ is $C^{\infty}$. Let $x_{0}\in X_{\Lambda}$ and fix a smooth transverse foliation $\mathcal{F}$ to $\mathcal{V}$ in $U = \mathcal{V}(W_{\rm loc}^{c}(x_{0}))$. Let $\pi_{x,y}^{\mathcal{F}}:\mathcal{V}(x)\to\mathcal{V}(y)$ be the $\mathcal{F}-$holonomy map between $\mathcal{V}(x)$ and $\mathcal{V}(y)$ (with $x,y\in U$ sufficiently close). Since the transverse $\mathcal{F}$ is $C^{\infty}$, the maps $\pi_{x,y}^{\mathcal{F}}$ are uniformly $C^{\infty}$ in $x$ and $y$. Let $x,y\in U$ be close and define an action
\begin{align}
\beta_{x,y}(\gamma) = \Phi_{x}\pi_{y,x}^{\mathcal{F}}\alpha(\gamma)\pi_{x,y}^{\mathcal{F}}\Phi_{x}^{-1}:\mathbb{T}^{d}\to\mathbb{T}^{d}.
\end{align}
Note that as $y\to x$ we have that $\beta_{x,y}(\gamma)\to\rho(\gamma)$ in the $C^{\infty}-$topology. From \cite{FisherMargulis2009}, if $x$ and $y$ are close then we obtain a $C^{\infty}-$conjugacy $H_{x,y}:\mathbb{T}^{d}\to\mathbb{T}^{d}$ such that $H_{x,y}\beta_{x,y}(\gamma) = \rho(\gamma)H_{x,y}$. Moreover, the conjugacy $H_{x,y}$ tends to ${\rm id}_{\mathbb{T}^{d}}$ in the $C^{\infty}-$topology as $y\to x$ since $\beta_{x,y}\to\rho$. On the other hand, we can also define
\begin{align}
\Tilde{H}_{x,y} = \Phi_{y}\pi_{x,y}^{\mathcal{F}}\Phi_{x}^{-1}
\end{align}
which satisfy
\begin{align}
\Tilde{H}_{x,y}\beta_{x,y}(\gamma) = \Phi_{y}\alpha(\gamma)\pi_{x,y}^{\mathcal{F}}\Phi_{x}^{-1} = \rho(\gamma)\Phi_{y}\pi_{x,y}^{\mathcal{F}}\Phi_{x}^{-1} = \rho(\gamma)\Tilde{H}_{x,y}.
\end{align}
Since the conjugacy $H_{x,y}$ is unique among all homeomorphisms close to identity (in the $C^{0}-$topology) and since $\Tilde{H}_{x,y}$ is close to identity when $x$ and $y$ are close we have $H_{x,y} = \Tilde{H}_{x,y}$. That is, we can write
\begin{align}
\Phi_{y} = H_{x,y}\Phi_{x}\pi_{y,x}^{\mathcal{F}}
\end{align}
which, since $H_{x,y}\to{\rm id}$ in $C^{\infty}$ and $\pi_{x,y}^{\mathcal{F}}$ is uniformly $C^{\infty}$, implies that $\Phi_{y}\to\Phi_{x}$ in the $C^{\infty}-$topology as $y\to x$. It follows that $\Phi:X_{\Lambda}\to\mathbb{T}^{d}$ is uniformly $C^{\infty}$ along $\mathcal{V}$. Since $\Phi$ is constant along $W^{c}$ it follows by Journé's lemma \cite{Journe1988} that $\Phi$ is a $C^{\infty}-$submersion.
\end{proof}
We are now ready to prove Theorem \ref{MainThm:AbelianCase}.
\begin{proof}[Proof of Theorem \ref{MainThm:AbelianCase}]
Fix a $\alpha-$invariant center leaf $W^{c}(x_{0})$. Let $H:X_{\Lambda}\to\mathbb{T}^{d}\times W^{c}(x_{0})$ be defined by
\begin{align}
H(x) := \left(\Phi(x),\pi^{\mathcal{V}}(x)\right)
\end{align}
where $\pi^{\mathcal{V}}$ is defined as the unique intersection between $\mathcal{V}(x)$ and $W^{c}(x_{0})$ (Lemma \ref{L:GlobalProductStructure}). Since $\mathcal{V}$ is a $C^{\infty}-$foliation (Lemma \ref{L:SmoothHorizontalFoliation}) the map $\pi^{\mathcal{V}}:X_{\Lambda}\to W^{c}(x_{0})$ is a $C^{\infty}-$submersion. By Lemma \ref{L:SmoothSubmersion} the map $\Phi$ is $C^{\infty}-$submersion. Combined it follows that $H$ is a diffeomorphism.

For any $\gamma\in\Gamma$ we have
\begin{align}
H(\alpha(\gamma)x) = \left(\Phi(\alpha(\gamma)x),\pi^{\mathcal{V}}(\alpha(\gamma)x)\right) = \left(\rho(\gamma)\Phi(x),\alpha(\gamma)\pi^{\mathcal{V}}(x)\right)
\end{align}
since $\alpha$ preserve $\mathcal{V}$. The action of $\alpha$ on $W^{c}(x_{0})$ is finite \cites{DeroinHurtado,Navas} and we can therefore identify $W^{c}(x_{0})$ with $\mathbb{T}$ such that $\alpha(\gamma)$ is given by translations. After this identification $H$ conjugates $\alpha$ to an affine action.
\end{proof}

\section{Smooth Rigidity: Accessible actions}\label{Sec:SmoothConjugacyAccessible}

In this section we prove Theorem \ref{MainThm:NonAbelianCase}. The main idea is similar to \cite{BrownRodriguez-HertzWang2017} but we need to overcome the issue that we do not have hyperbolic matrices in the representation on the fundamental group since the action contain no Anosov diffeomorphism.

Firstly, we may assume that the conclusion of \Cref{L:RestrictToAlgebraicGroup} holds from now on. That is,
$\mathbf{G}$ is an algebraically simply connected, algebraic group defined over $\mathbb{R}$. Assume that all $\mathbb{R}-$simple factors have $\mathbb{R}-$rank $2$ or higher. Let $\Gamma$ be a lattice in $G=\mathbf{G}(\mathbb{R})^{\circ}$. 

In this setting, we can apply the following version of Zimmer's cocycle superrigidity theorem:

\begin{theorem}[Zimmer's cocycle superrigidity, \cite{MR2039990}]\label{thm:ZCSR}
 Let $(X,\mu)$ be a ergodic $\Gamma$ Lebesgue space. Assume that $\mu$ is $\Gamma$ invariant. Let $\beta:\Gamma\times X\to{\rm GL}(n,\mathbb{R})$ be a measurable cocycle over the $\Gamma$ action on $X$. Further assume that \begin{equation}\tag{L1}\label{L1int}\beta(\gamma,\cdot)\in L^{1}(X,\mu)\end{equation} for all $\gamma\in \Gamma$. 
    
    Then, there exists
    \begin{enumerate}
        \item a measurable map $\phi:X\to{\rm GL}(n,\mathbb{R})$,
        \item a continuous representaiton $\pi: G\to{\rm SL}(n,\mathbb{R})$,
        \item a compact subgroup $K$ in ${\rm GL}(n,\mathbb{R})$ that commutes with $\pi(G)$, and
        \item a measurable cocycle $\kappa:\Gamma\times X\to K$
    \end{enumerate}
    such that 
    \[\beta(\gamma,x)=\phi(\gamma.x)\pi(\gamma)\kappa(\gamma,x)\phi(x)^{-1}\] for all $\gamma\in \Gamma$ and $\mu$ almost every $x\in X$.
\end{theorem}
Note that since we will apply \Cref{thm:ZCSR} to derivative cocycle over a smooth $\Gamma$ action on a compact manifold, the integrability condition \cref{L1int} will always be satisfied.

We now start the proof of Theorem \ref{MainThm:NonAbelianCase} (following \cite{BrownRodriguez-HertzWang2017}). The idea is to produce a higher rank subgroup $\mathbb{Z}^{k}\cong\Sigma\leq\Gamma$ with $\gamma\in\Sigma$ such that $\alpha(\gamma)$ is partially hyperbolic with $1-$dimensional center. Theorem \ref{MainThm:NonAbelianCase} then follows from Theorem \ref{Thm:TopologicalRigidity} and \cite{Sandfeldt2024}*{Theorem 1.2}. To produce a higher rank (abelian group) action with a partially hyperbolic element we follow \cite{BrownRodriguez-HertzWang2017}*{Section 8}.
\begin{lemma}\label{L:DefinitionOfZariskiDenseSet}
Let $\pi:\Gamma\to1\times{\rm GL}(d,\mathbb{R})\subset{\rm GL}(d+1,\mathbb{R})$ be a representation such that $\pi(\gamma_{0})$ is partially hyperbolic with $1-$dimensional isometric center and let
\begin{align}
W_{\pi} := \{w\in G\text{ : }\pi(w)E_{\gamma_{0}}^{s,\pi}\cap E_{\gamma_{0}}^{u,\pi} = E_{\gamma_{0}}^{s,\pi}\cap \pi(w)E_{\gamma_{0}}^{u,\pi} = 0\}.
\end{align}
The set
\begin{align}
W := \bigcap_{\pi:\Gamma\to1\times{\rm GL}(d,\mathbb{R})}W_{\pi}
\end{align}
is Zariski open and Zariski dense.
\end{lemma}
\begin{proof}
This follows since there are, up to conjugation, only finitely many representations $\pi:\Gamma\to{\rm GL}(d+1,\mathbb{R})$ (see also \cite{BrownRodriguez-HertzWang2017}*{Lemma 8.11}).
\end{proof}
Our next goal is to show that for any $\gamma\in W\cap\Gamma$ we have $D\alpha(\gamma)E^{s}(x)\cap E^{u}(\alpha(\gamma)x) = E^{s}(\alpha(\gamma)x)\cap D\alpha(\gamma)E^{u}(x) = 0$. This is essentially (the proof of) \cite{BrownRodriguez-HertzWang2017}*{Proposition 8.7}, but the proof has to be slightly altered since our action does not contain an Anosov diffeomorphism (which implies that there are uncountably many ergodic $\alpha-$invariant measures).
\begin{lemma}\label{L:TransversalityLemma}
If $\gamma\in W\cap\Gamma$ then $\alpha(\gamma)^{*}E^{s}$ is transverse to $E^{u}$ (in $E^{s}\oplus E^{u}$) and $\alpha(\gamma)^{*}E^{u}$ is transverse to $E^{s}$ (in $E^{s}\oplus E^{u}$).
\end{lemma}
Fix $x\in X_{\Lambda}$. Recall that $\gamma_{0}\in\Gamma$ is such that $f = \alpha(\gamma_{0})$ is partially hyperbolic with $1-$dimensional center. We denote by $K_{x} = \overline{\alpha(\Gamma)x}$ which is $\alpha-$invariant. The following lemma is elementary.
\begin{lemma}\label{L:PointLieInOpenInvSet}
If $U\subset K_{x}$ is relatively open, non-empty, and $\alpha-$invariant then $x\in U$.
\end{lemma}
\begin{proof}
Let $C = K_{x}\setminus U$, then $C$ is compact and $\alpha-$invariant. If $x\in C$ then $\alpha(\Gamma)x\subset C$ so $K_{x} = \overline{\alpha(\Gamma)x}\subset C$ which is a contradiction since $U$ is non-empty.
\end{proof}
\begin{lemma}
There is a $\alpha-$invariant ergodic probability measure $\mu_{x}$ with support $K_{x}$.
\end{lemma}
\begin{proof}
Let $H$ be the conjugacy from \Cref{Thm:TopologicalRigidity}. We have $HK_{x} = H\overline{\alpha(\Gamma)x} = \overline{H\alpha(\Gamma)x} = \overline{\Hat{\rho}(\Gamma)H(x)}$, so $HK_{x}$ is a $\Hat{\rho}-$orbit closure. By \cite{BrownRodriguez-HertzWang2017}*{Proposition 6.5} the set $HK_{x}$ is homogeneous and can be written as a finite union of nilsubmanifolds of $X_{\Lambda}$. It is clear that the element $\Hat{\rho}(\gamma_{0})$ acts ergodically on $HK_{x}$ with respect to the $\Hat{\rho}-$invariant volume $\nu$ (if $\mathfrak{p}$ is the Lie algebra tangent to $HK_{x}$ then either $E_{0}^{c} = [\mathfrak{p},\mathfrak{p}]$ in which case $\Hat{\rho}(\gamma_{0})$ is ergodic with respect to the volume on $HK_{x}$ or $[\mathfrak{p},\mathfrak{p}] = 0$ in which case $HK_{x}$ does not contain $[N,N]$ so $\Hat{\rho}(\gamma_{0})$ is hyperbolic on $HK_{x}$). With $\mu_{x} = H_{*}^{-1}\nu$ the lemma follows.
\end{proof}
Let $D\alpha:\Gamma\times X_{\Lambda}\to{\rm GL}(\mathfrak{n})$ be the derivative cocycle of $\alpha$ (where we have used that the tangent bundle of $X_{\Lambda}$ can naturally be identified with $X_{\Lambda}\times\mathfrak{n}$). By Zimmer's superrigidity (Theorem \ref{thm:ZCSR}) there is a representation $\pi:\Gamma\to{\rm GL}(\mathfrak{n})$, a measurable map $C:X_{\Lambda}\to{\rm GL}(\mathfrak{n})$, and a $\alpha-$cocycle taking values in a compact group $K:\Gamma\times X_{\Lambda}\to{\rm SO}(\mathfrak{n})$ such that
\begin{align}
D_{x}\alpha(\gamma) = C(\alpha(\gamma)x)\pi(\gamma)K(\gamma,x)C(x)^{-1},\quad\mu_{x}-{\rm a.s}.
\end{align}
Note that $\pi(\gamma_{0})$ is partially hyperbolic with $1-$dimensional center. Moreover, the space $E^{c}$ is $\alpha-$invariant so the representation $\pi = \pi'\oplus 1$ has an identity factor (coinciding with $E^{c}$ i.e. $E^{c} = C(x)V_{1}$ where $\mathfrak{n} = V_{\pi} = V_{1}\oplus V_{\pi'}$). That is, $\pi$ satisfies the assumptions of Lemma \ref{L:DefinitionOfZariskiDenseSet}. This also shows that
\begin{align}
\alpha(\gamma)^{*}(E^{s}\oplus E^{u}) = E^{s}\oplus E^{u},\quad\gamma\in\Gamma
\end{align}
since $C(x)V_{\pi'}$ is $\alpha(\gamma_{0})-$invariant and transverse to $E^{c}$, but $E^{s}\oplus E^{u}$ is unique with this property. Let $\{\gamma_{1},\gamma_{2},...\} = \Gamma\setminus\{\gamma_{0}\}$ be an enumeration of $\Gamma\setminus\{\gamma_{0}\}$. Define $\eta_{j} = \gamma_{j}\gamma_{0}\gamma_{j}^{-1}$ and note that each $\pi(\eta_{j})$ is partially hyperbolic with $1-$dimensional center since $\pi(\eta_{j})$ is conjugated to $\pi(\gamma_{0})$. Denote by $E_{\eta_{j}}^{\sigma,\pi}$, $\sigma = s,u$, the stable or unstable space of $\pi(\eta_{j})$. It is clear that $\pi(\gamma_{j})E_{\gamma_{0}}^{\sigma,\pi} = E_{\eta_{j}}^{\sigma,\pi}$. Let
\begin{align}
\mathbf{S}^{\pi,r} = \bigcap_{\substack{j = 0,...,r \\ \sigma = s,u}}{\rm Stab}(E_{\eta_{j}}^{\sigma,\pi}) = \bigcap_{\substack{j = 0,...,r \\ \sigma = s,u}}{\rm Stab}(E_{\eta_{j}}^{\sigma,\pi}) = \bigcap_{\substack{j = 0,...,r \\ \sigma = s,u}}{\rm Stab}(\pi(\gamma_{j})E_{\gamma_{0}}^{\sigma,\pi})
\end{align}
then $\mathbf{S}^{\pi,r+1}\subset\mathbf{S}^{\pi,r}$. Moreover, each $\mathbf{S}^{\pi,r}$ is an algebraic subgroup of $G$ so the sequence $\mathbf{S}^{\pi,r}$ terminates \cite{BrownRodriguez-HertzWang2017}*{Section 8.3}. It follows that there is some $r_{0}(\pi)$ such that
\begin{align}
\mathbf{S}^{\pi,r_{0}(\pi)} = \mathbf{S}^{\pi} := \bigcap_{\substack{\gamma\in\Gamma \\ \sigma = s,u}}{\rm Stab}(\pi(\gamma_{j})E_{\gamma_{0}}^{\sigma,\pi}) = \bigcap_{\substack{g\in G \\ \sigma = s,u}}{\rm Stab}(\pi(g)E_{\gamma_{0}}^{\sigma,\pi})
\end{align}
where the last equality follows since $\Gamma$ is Zariski dense in $G$ \cite{BrownRodriguez-HertzWang2017}*{Section 8.3}. We also define $E_{j}^{\sigma}$, $\sigma = s,u$, to be the invariant distributions associated to $\alpha(\eta_{j})$ (note that $\alpha(\eta_{j})$ is conjugated to $f = \alpha(\gamma_{0})$, so $\alpha(\eta_{j})$ is in particular partially hyperbolic with $1-$dimensional center). Let 
\begin{align}
& d_{s} = \dim E^{s}(x) = \dim E_{\gamma_{0}}^{s,\pi} = \dim E_{\eta_{j}}^{s,\pi}, \\
& d_{u} = \dim E^{u}(x) = \dim E_{\gamma_{0}}^{u,\pi} = \dim E_{\eta_{j}}^{u,\pi},
\end{align}
and denote by ${\rm Gr}(d_{s},\mathfrak{n})$, ${\rm Gr}(d_{u},\mathfrak{n})$ the Grassmanians with $d_{\sigma}-$dimensional subspaces. Let
\begin{align*}
\Phi:{\rm GL}(\mathfrak{n})\times\left[{\rm Gr}(d_{s},\mathfrak{n})^{r_{0}(\pi)}\times{\rm Gr}(d_{u},\mathfrak{n})^{r_{0}(\pi)}\right]\to{\rm Gr}(d_{s},\mathfrak{n})^{r_{0}(\pi)}\times{\rm Gr}(d_{u},\mathfrak{n})^{r_{0}(\pi)}
\end{align*}
be the standard action. Define a map $\tau:X_{\Lambda}\to{\rm Gr}(d_{s},\mathfrak{n})^{r_{0}(\pi)}\times{\rm Gr}(d_{u},\mathfrak{n})^{r_{0}(\pi)}$ by
\begin{align}
\tau(x) = \left((E_{j}^{s})_{j = 0,...,r_{0}(\pi)},(E_{j}^{u})_{j = 0,...,r_{0}(\pi)}\right)
\end{align}
it is then clear that $\tau$ is a continuous map. It is also clear that $C(x)E_{\eta_{j}}^{\sigma,\pi} = E_{j}^{\sigma}(x)$ (since $K(\gamma,x)$ commute with $\pi(\gamma)$) so 
\begin{align}
\Phi\left(C(y),\left((E_{\eta_{j}}^{s,\pi})_{j = 0,...,r_{0}(\pi)},(E_{\eta_{j}}^{u,\pi})_{j = 0,...,r_{0}(\pi)}\right)\right) = \tau(y)\quad\mu_{x}-{\rm a.s}.
\end{align}
It follows that 
\begin{align}
\tau({\rm supp}(\mu_{x}))\subset\overline{{\rm Orb}_{\Phi}\left((E_{\eta_{j}}^{s,\pi})_{j = 0,...,r_{0}(\pi)},(E_{\eta_{j}}^{u,\pi})_{j = 0,...,r_{0}(\pi)}\right)}.
\end{align}
Since $\tau(y)\in{\rm Orb}_{\Phi}\left((E_{\eta_{j}}^{s,\pi})_{j = 0,...,r_{0}(\pi)},(E_{\eta_{j}}^{u,\pi})_{j = 0,...,r_{0}(\pi)}\right)$ for $\mu_{x}-$almost every $y$ and since ${\rm Orb}_{\Phi}\left((E_{\eta_{j}}^{s,\pi})_{j = 0,...,r_{0}(\pi)},(E_{\eta_{j}}^{u,\pi})_{j = 0,...,r_{0}(\pi)}\right)$ is open in 
\begin{align}
\overline{{\rm Orb}_{\Phi}\left((E_{\eta_{j}}^{s,\pi})_{j = 0,...,r_{0}(\pi)},(E_{\eta_{j}}^{u,\pi})_{j = 0,...,r_{0}(\pi)}\right)},
\end{align}
see \cite{BrownRodriguez-HertzWang2017}*{Lemma 8.8}, it follows that 
\begin{align}
U^{\pi} := \tau^{-1}{\rm Orb}_{\Phi}\left((E_{\eta_{j}}^{s,\pi})_{j = 0,...,r_{0}(\pi)},(E_{\eta_{j}}^{u,\pi})_{j = 0,...,r_{0}(\pi)}\right)
\end{align}
is open and dense in ${\rm supp}(\mu_{x})$. Define $\Hat{C}:U^{\pi}\to{\rm GL}(\mathfrak{n})/S^{\pi}$ by
\begin{align}
\Phi\left(\Hat{C}(x),\left((E_{\eta_{j}}^{s,\pi})_{j = 0,...,r_{0}(\pi)},(E_{\eta_{j}}^{u,\pi})_{j = 0,...,r_{0}(\pi)}\right)\right) = \tau(x),
\end{align}
it is immediate that $\Hat{C}$ is continuous since $\Phi$ and $\tau$ are both continuous.

The following lemma is \cite{BrownRodriguez-HertzWang2017}*{Lemma 8.9}.
\begin{lemma}\label{L:PropertiesOfBundleParametrization}
The set $U^{\pi}$ is $\alpha-$invariant and for $y\in U^{\pi}$ we have
\begin{align}
D_{y}\alpha(\gamma)E^{\sigma}(y) = D_{y}\alpha(\gamma)\Hat{C}(y)E_{\gamma_{0}}^{\sigma,\pi} = \Hat{C}(\alpha(\gamma)y)\pi(\gamma)E_{\gamma_{0}}^{\sigma,\pi}.
\end{align}
\end{lemma}
We are now ready to prove Lemma \ref{L:TransversalityLemma}.
\begin{proof}[Proof of Lemma \ref{L:TransversalityLemma}]
Since $U^{\pi}$ is relatively open in $K_{x}$, $\alpha-$invariant, and non-empty (by Lemma \ref{L:PropertiesOfBundleParametrization}) it follows from Lemma \ref{L:PointLieInOpenInvSet} that $x\in U^{\pi}$ (this also implies that $\alpha(\gamma)x\in U^{\pi}$ for all $\gamma\in\Gamma$ since $U^{\pi}$ is $\alpha-$invariant). Let $\gamma\in W\cap\Gamma$. By the assumption on $\gamma$
\begin{align}
\pi(\gamma)E_{\gamma_{0}}^{s,\pi}\cap E_{\gamma_{0}}^{u,\pi} = E_{\gamma_{0}}^{s,\pi}\cap \pi(\gamma)E_{\gamma_{0}}^{u,\pi} = 0
\end{align}
which immediately implies (Lemma \ref{L:PropertiesOfBundleParametrization}) that
\begin{align*}
0 = & \Hat{C}(\alpha(\gamma)x)\left(\pi(\gamma)E_{\gamma_{0}}^{s,\pi}\cap E_{\gamma_{0}}^{u,\pi}\right) = \left(\Hat{C}(\alpha(\gamma)x)\pi(\gamma)E_{\gamma_{0}}^{s,\pi}\right)\cap\left(\Hat{C}(\alpha(\gamma)x)E_{\gamma_{0}}^{u,\pi}\right) = \\ &
= \left(D_{x}\alpha(\gamma)E^{s}(x)\right)\cap\left(E^{u}(\alpha(\gamma)x)\right)
\end{align*}
and similarly
\begin{align*}
\left(D_{x}\alpha(\gamma)E^{u}(x)\right)\cap\left(E^{s}(\alpha(\gamma)x)\right) = 0.
\end{align*}
Since $x$ was arbitrary the lemma follows.
\end{proof}
Let $C_{\varepsilon}^{\sigma}(x)$ be the $\varepsilon-$cone about the distribution $E^{\sigma}(x)$, $\sigma = s,u$. By applying the proof of \cite{BrownRodriguez-HertzWang2017}*{Proposition 8.12} in the bundle $E^{s}\oplus E^{u}$ (which we recall is $\alpha-$invariant) we obtain the following lemma from \Cref{L:TransversalityLemma}.
\begin{lemma}\label{L:ConeConditionFromTrans}
Fix $\varepsilon > 0$ small. If $\gamma\in W\cap\Gamma$ (with $W$ as in Lemma \ref{L:TransversalityLemma}) then there is $N\geq 0$ such that $g = f^{N}\alpha(\gamma)f^{N}$ satisfy
\begin{align}
Dg^{-1}(C_{\varepsilon}^{s}(x))\subset C_{\varepsilon/2}^{s}(g^{-1}x),\quad Dg(C_{\varepsilon}^{u}(x))\subset C_{\varepsilon/2}^{u}(gx)\quad x\in X_{\Lambda}.
\end{align}
Moreover, for each $v^{s}\in C_{\varepsilon}^{s}(x)$ we have $\norm{Dg(v^{s})}\leq\norm{v^{s}}/2$, for each $v^{u}\in C_{\varepsilon}^{u}(x)$ we have $\norm{Dg^{-1}(v^{u})}\leq\norm{v^{u}}/2$, and $g$ is partially hyperbolic with $1-$dimensional center.
\end{lemma}
\begin{proof}
The first part of the lemma follows as in \cite{BrownRodriguez-HertzWang2017}*{Proposition 8.12}. That $g = f^{N}\alpha(\gamma)f^{N}$ gives a hyperbolic cocycle when restricted to $E^{s}\oplus E^{u}$ then follows from the standard cone criterion.

Note that \textbf{any} $\alpha(\gamma)$ is uniformly subexponential along $E^{c}$ by Theorem \ref{Thm:TopologicalRigidity} (in fact isometric, which is clear from the proof), since $Dg$ is hyperbolic along $E^{s}\oplus E^{u}$ it follows that $g$ is partially hyperbolic with $1-$dimensional center.
\end{proof}
Let $S\subset\Gamma$ be the subset defined by $\gamma\in S$ if:
\begin{align}
& D\alpha(\gamma)^{-1}C_{\varepsilon}^{s}(x)\subset C_{\varepsilon}^{s}(\alpha(\gamma)x), \\ 
& \norm{D\alpha(\gamma)v^{s}}\leq\norm{v^{s}}/2,\quad v^{s}\in C_{\varepsilon}^{s}(x), \\
& D\alpha(\gamma)C_{\varepsilon}^{u}(x)\subset C_{\varepsilon}^{u}(\alpha(\gamma)x), \\ 
& \norm{D\alpha(\gamma)^{-1}v^{u}}\leq\norm{v^{u}}/2,\quad v^{u}\in C_{\varepsilon}^{u}(x).
\end{align}
The following lemma is \cite{BrownRodriguez-HertzWang2017}*{Lemma 8.6}.
\begin{lemma}\label{L:BasicPropertiesOfPHElements}
The set $S$ is a semigroup, every $\alpha(\gamma)$ ($\gamma\in S$) is partially hyperbolic with $1-$dimensional center, and there is some $N > 0$ such that $\gamma_{0}^{N}\in S$.
\end{lemma}
The following is immediate from Lemmas \ref{L:ConeConditionFromTrans} and \ref{L:BasicPropertiesOfPHElements} (see \cite{BrownRodriguez-HertzWang2017}*{Proposition 8.5}).
\begin{lemma}\label{L:ZariskiDensePH}
The set $S$ is Zariski dense.
\end{lemma}
We are now ready to prove Theorem \ref{MainThm:NonAbelianCase}.
\begin{proof}[Proof of Theorem \ref{MainThm:NonAbelianCase}]
By Lemmas \ref{L:ZariskiDensePH} and \ref{L:BasicPropertiesOfPHElements} the set $S\subset\Gamma$ is a Zariski dense semigroup and each $\alpha(\gamma)$ ($\gamma\in S$) is partially hyperbolic with $1-$dimensional center along $E^{c}$. It follows that $\rho(\gamma)$ is hyperbolic for $\gamma\in S$ (where we recall that $\Phi(\alpha(\gamma)x) = \rho(\gamma)\Phi(x)$). By \Cref{L:ExistenceOfHigherRankAbelianGroup} there is $\mathbb{Z}^{k}\cong\Sigma$ such that $\rho|_{\Sigma}:\Sigma\to{\rm GL}(d,\mathbb{Z})$ is higher rank and at least one $\gamma\in\Sigma\cap S$ (so $\alpha(\gamma)$ is partially hyperbolic with $1-$dimensional center). It follows that the action $\alpha|_{\Sigma}:\Sigma\times X_{\Lambda}\to X_{\Lambda}$ is a higher rank abelian action containing a partially hyperbolic element. Since the partially hyperbolic element has $1-$dimensional center, $N$ is not abelian, and $\alpha$ is conjugated to an affine action (\Cref{Thm:TopologicalRigidity}) the partially hyperbolic element is accessible. By \cite{Sandfeldt2024}*{Theorem 7.1} the conjugacy $H$ (from \Cref{Thm:TopologicalRigidity}) is bi-Hölder. Applying \cite{Sandfeldt2024}*{Theorem 1.2} to the action $\alpha|_{\Sigma}:\Sigma\times X_{\Lambda}\to X_{\Lambda}$ finishes the proof of Theorem \ref{MainThm:NonAbelianCase}.
\end{proof}

\newpage

\appendix

\section{The QI-condition and existence of a fibered structure}\label{Appendix:QI-Condition}

In our main results, Theorems \ref{MainThm:AbelianCase} and \ref{MainThm:NonAbelianCase}, we assume that the partially hyperbolic element $f$ is QI. The main consequence of $f$ being QI is that this implies that $f$ is fibered (as in Section \ref{SubSec:FiberedPH}). Here we show that, conversely, if $f$ is fibered (with some mild assumptions on the fibered structure) then $f$ is also QI. We will assume that $f$ is homotopic to a partially hyperbolic automorphism $L$ with $1-$dimensional center (and if $X_{\Lambda}=\mathbb{T}^{d+1}$ then we will assume that $L = A\times 1$ with $A$ hyperbolic). Let $\Phi:X_{\Lambda}\to\mathbb{T}^{d}$ be the map from Theorem \ref{Thm:PropertiesOfFibration} (this map always exists, but may not be a fiber bundle \cite{Sandfeldt2024}*{Lemma 3.1}).
\begin{lemma}\label{L:FiberedStructureGiveQI}
If $\Phi:X_{\Lambda}\to\mathbb{T}^{d}$ makes $f$ into a fibered partially hyperbolic diffeomorphism then $f$ is ${\rm QI}$.
\end{lemma}
\begin{proof}
Let $\Phi:N\to\mathbb{R}^{d}$ be a lift such that for some lift $F$ of $f$ we have $\Phi(Fx) = L_{su}\Phi(x)$. Denote by $\Tilde{W}^{s}$, $\Tilde{W}^{u}$ the stable and unstable foliations lifted to $N$. Note that $\Phi:\Tilde{W}^{\sigma}(x)\to\Phi(x) + E_{0}^{\sigma}$, $\sigma = s,u$, is a homeomorphism for every $x\in N$. Let $r,r' > 0$ be such that $\Phi(\Tilde{W}_{r}^{\sigma}(x))\supset\Phi(x) + B_{r'}^{\sigma}(0)$, where $B_{r'}^{\sigma}(0)$ is the $r'-$ball around $0$ in $E_{0}^{\sigma}$, for all $x\in N$. Let $y\in\Tilde{W}^{\sigma}(x)$ and consider $p = \Phi(x)$, $q = \Phi(y)$. We find $p_{1},...,p_{N}\in p + E_{0}^{\sigma}$ such that $N\leq\intd(p,q)/r' + 1$, $p_{j+1}\in B_{r'}^{\sigma}(p_{j})$ for $j < N$, and $p_{N} = q$, $p_{1} = p$. Let $x_{j}\in\Tilde{W}^{\sigma}(x)$ be the unique point such that $\Phi(x_{j}) = p_{j}$. Since $\Phi(\Tilde{W}_{r}^{\sigma}(x_{j}))\supset p_{j} + B_{r'}^{\sigma}(0) = B_{r'}^{\sigma}(p_{j})$ we have $x_{j+1}\in\Tilde{W}_{r}^{\sigma}(x_{j})$. It follows that we can estimate
\begin{align}\label{Eq:KeyEstimate}
\intd_{\sigma}(x,y)\leq\sum_{j = 1}^{N-1}\intd_{\sigma}(x_{j},x_{j+1})\leq rN\leq\frac{r}{r'}\cdot\intd(\Phi(x),\Phi(y)) + r.
\end{align}
Write $\Phi(x) = \pi(x) + \varphi(x)$ for some $\varphi:X_{\Lambda}\to\mathbb{R}^{d}$. It follows that
\begin{align*}
\intd(\Phi(x),\Phi(y)) = &  \intd(\pi(x) + \varphi(x),\pi(y) + \varphi(y))\leq \\ &
\intd(\pi(x),\pi(y)) + 2\norm{\varphi}_{C^{0}}\leq\intd(x,y) + 2\norm{\varphi}_{C^{0}}
\end{align*}
where the last inequality uses that $\pi$ is a Riemannian submersion (after choosing appropriate left-invariant metrics on $N$ and $\mathbb{R}^{d}$). If we define $A = r/r'$ and $B = 2\norm{\varphi}_{C^{0}}r/r' + r$ then Equation \ref{Eq:KeyEstimate} implies
\begin{align}
\intd_{\sigma}(x,y)\leq A\intd(x,y) + B.
\end{align}
This proves the lemma since $\intd_{\sigma}$ is comparable to $\intd$ for $x$ and $y$ close (the leaves of $\Tilde{W}^{\sigma}$ are uniformly $C^{1}$).
\end{proof}
As a corollary we obtain conditions that guarantee that $f$ is QI.
\begin{lemma}\label{L:SufficientConditionForQI}
Any one of the following conditions implies that $f$ is ${\rm QI}$-partially hyperbolic:
\begin{enumerate}[label = (\roman*)]
    \item $f$ is topologically conjugated to $L$,
    \item $f$ is leaf-conjugated to $L$,
    \item there is a fiber bundle $p:X_{\Lambda}\to B$ with the fibers of $p$ coinciding with the center leaves of $f$ and $B$ a nilmanifold.
\end{enumerate}
\end{lemma}
\begin{remark}
It seems likely that one can remove the assumption that $B$ is a nilmanifold in $(iii)$. However, note that $p$ collapses the center leaves so morally $p$ should semi-conjugate $f$ to an Anosov map which, conjecturally, implies that $B$ is an (infra-)nilmanifold (though, of course, there is no guarantee that $p$ semi-conjugates $f$ to a smooth map). Moreover, as is clear in the proof, we only need to assume that $B$ is aspherical (or, even more generally, that $\pi_{2}B = 1$).
\end{remark}
\begin{proof}
Note that $(iii)$ implies both $(ii)$ and $(i)$, so it suffices to prove $(iii)$. The fibers of $p$ are $1-$dimensional and compact, so the fibers are circles. Since $X_{\Lambda}$, $B$ and the fibers of $p$ are all aspherical (this is an immediate consequence of the long exact sequence of homotopy, noting that any nilmanifold is a torus bundle over a nilmanifold of lower step), the long exact sequence of homotopy gives a short exact sequence
\begin{align}\label{Eq:LongExactSequence}
1\to\mathbb{Z}\to\Lambda\to\pi_{1}B\to 1
\end{align}
and $\pi_{n}B = 1$ for $n\geq 1$. From Equation \ref{Eq:LongExactSequence} (and the fact that $\ker(\Lambda\to\pi_{1}B)$ is $f_{*}-$invariant) it is immediate that $\pi_{1}B\cong\mathbb{Z}^{d}$. Since $\pi_{1}B = \mathbb{Z}^{d}$ and $\pi_{n}B = 1$ for $n\geq 2$, $B$ is a $K(1,\mathbb{Z}^{d})-$space. It follows that $B$ is homeomorphic to the torus (see for example \cite{Stark2001}). Since $p$ semi-conjugates $f$ to an Anosov homeomorphism on the torus $B$ it follows that, after possibly composing with a homeomorphism of $B$ \cite{Sumi1996}*{Theorem 2(1)} (or \cite{Doucette2023}*{Theorem E}), we may assume that $p:X_{\Lambda}\to\mathbb{T}^{d}$ semi-conjugates $f$ to a hyperbolic automorphism. After composing $p$ with an automorphism of $\mathbb{T}^{d}$ we may also assume that $p$ is homotopic to the standard projection $\pi:X_{\Lambda}\to\mathbb{T}^{d}$. However, the uniqueness of the map $\Phi:X_{\Lambda}\to\mathbb{T}^{d}$ (see \cite{Sandfeldt2024}*{Lemma 3.1}) then implies that $\Phi = p$. Lemma \ref{L:SufficientConditionForQI} now follows from Lemma \ref{L:FiberedStructureGiveQI}.
\end{proof}

\newpage
\bibliography{main.bib}{}
\bibliographystyle{abbrv}

\end{document}